\documentclass[preprint]{imsart}

\usepackage[svgnames]{xcolor}
\usepackage{amsfonts}
\usepackage{amsmath}
\usepackage{amsthm}
\usepackage{floatrow}
\usepackage{float}
\usepackage{enumerate}
\usepackage{algorithm}
\usepackage{algpseudocode}
\usepackage{listings}
\usepackage{catchfilebetweentags}
\usepackage[english]{babel}
\usepackage{bbm}
\usepackage{graphicx}
\usepackage{subcaption}
\usepackage{ragged2e}
\usepackage[font=footnotesize]{caption}
\usepackage{hhline}
\usepackage{tikz}
\usetikzlibrary{backgrounds,trees,arrows}
\RequirePackage[sort,numbers]{natbib}
\RequirePackage[colorlinks,citecolor=blue,urlcolor=blue]{hyperref}

\usepackage{comment}

\startlocaldefs
\theoremstyle{plain}
\newtheorem{thm}{Theorem}[section]
\newtheorem{prop}[thm]{Proposition}
\newtheorem{lemma}[thm]{Lemma}
\newtheorem{cor}[thm]{Corollary}

\theoremstyle{definition}

\newtheorem{defi}[thm]{Definition}
\newtheorem{ex}{Example}

\theoremstyle{remark}
\newtheorem{remark}[thm]{Remark}

\newcommand{\cithree}[3]{_{{#1}\;\!}{#2}_{\;{#3}}}
\newcommand{\q}[1]{``#1''}

\DeclareMathOperator*{\argmax}{\arg\max}
\newcommand{\norm}[1]{\left\lVert#1\right\rVert}
\algnewcommand{\IfThenElse}[3]{% \IfThenElse{<if>}{<then>}{<else>}
  \State \algorithmicif\ #1\ \algorithmicthen\ #2\ \algorithmicelse\ #3}
  \algnewcommand{\IfThen}[2]{% \IfThen{<if>}{<then>}
  \State \algorithmicif\ #1\ \algorithmicthen\ #2}
\algnewcommand{\InlineFor}[2]{%
    \State \algorithmicfor\ {#1}\ \algorithmicdo\ {#2} \algorithmicend\ \algorithmicfor%
}
\algnewcommand{\LineComment}[1]{\Statex \(\triangleright\) #1}
\algnewcommand\algorithmicforeach{\textbf{for each}}
\algdef{S}[FOR]{ForEach}[1]{\algorithmicforeach\ #1\ \algorithmicdo}

\tikzset{
    invisible/.style={opacity=0},
    visible on/.style={alt=#1{}{invisible}},
    alt/.code args={<#1>#2#3}{%
      \alt<#1>{\pgfkeysalso{#2}}{\pgfkeysalso{#3}} % \pgfkeysalso doesn't change the path
    },
  }
  
\makeatletter
\def\@maketitle{%
  \newpage
  \null
  \vskip 2em%
  \begin{center}%
  \let \footnote \thanks
    {\Large\bfseries \@title \par}%
    \vskip 1.5em%
    {\normalsize
      \lineskip .5em%
      \begin{tabular}[t]{c}%
        \@author
      \end{tabular}\par}%
    \vskip 1em%
    {\normalsize \@date}%
  \end{center}%
  \par
  \vskip 1.5em}
\makeatother
\endlocaldefs

\begin{document}

\begin{frontmatter}

\title{From the Bernoulli Factory to a Dice Enterprise via Perfect Sampling of Markov Chains}
\runtitle{A Dice Enterprise}

\begin{aug}
\author{\fnms{Giulio} \snm{Morina}\ead[label=e1]{G.Morina@warwick.ac.uk}}
\and
\author{\fnms{Krzysztof} \snm{\L atuszy\'nski}\ead[label=e2]{K.G.Latuszynski@warwick.ac.uk}}

\address{
Department of Statistics\\
University of Warwick\\
Coventry, CV4 7AL - UK \\
\printead{e1,e2}}

\author{\fnms{Piotr} \snm{Nayar}
\ead[label=e3]{P.Nayar@mimuw.edu.pl}}

\address{
Institute of Mathematics  \\
Faculty of Mathematics, Informatics and Mechanics \\
University of Warsaw \\
Banacha 2, 02-097 Warsaw - Poland \\
\printead{e3}\\}

\author{\fnms{Alex} \snm{Wendland}
\ead[label=e4]{a.p.wendland@gmail.com}}

\address{
Mathematics Institute \\
Zeeman Building \\
University of Warwick \\
Coventry, CV4 7AL – UK \\
\printead{e4}\\}

\runauthor{G. Morina et al.}

\end{aug}

\begin{abstract}
\ Given a $p$-coin that lands heads with unknown probability $p$, we
wish to produce an $f(p)$-coin for a given function $f: (0,1)
\rightarrow (0,1)$. This problem is commonly known as the Bernoulli
Factory and results on its solvability  and complexity have been
obtained in \cite{Keane1994,Nacu2005}. Nevertheless, generic ways to
design a practical Bernoulli Factory for a given function $f$ exist
only in a few special cases. We present a constructive way to build an
efficient Bernoulli Factory when $f(p)$ is a rational function with
coefficients in $\mathbb{R}$. Moreover, we extend the Bernoulli
Factory problem to a more general setting where we have access to an
$m$-sided die and we wish to roll a $v$-sided one; i.e., we 
consider rational functions between
open probability simplices. Our construction consists of rephrasing
the original problem as simulating from the stationary distribution of
a certain class of Markov chains - a task that we show can be achieved
using perfect simulation techniques with the original $m$-sided die as
the only source of randomness. In the Bernoulli Factory case, the number of tosses
needed by the algorithm has exponential tails and its expected value can be bounded uniformly in $p$. En route
to optimizing the algorithm we show a fact of independent interest:
every finite, integer valued, random variable will eventually become log-concave
after convolving with enough Bernoulli trials.  
\end{abstract}
\end{frontmatter}

\newpage
 \section{Introduction}

Back in 1951, Von Neumann \cite{VonNeumann1951} proposed a method to
produce a fair coin out of a biased one. Since then, the problem has
been generalized into finding an algorithm that given a $p$-coin --- a
coin that lands heads with unknown probability $p$ --- can produce an
$f(p)$-coin for a given function $f: \mathcal{D} \subseteq (0,1)
\rightarrow (0,1)$. Keane and O'Brien \cite{Keane1994} referred to
this problem as the Bernoulli Factory and, motivated by problems in
regenerative steady-state simulations \cite{asmussen1992stationarity, MR1957197}, identified the class of
functions $f$ for which it is solvable. Since then, other studies have
been carried out to provide ways of constructing and analysing the
Bernoulli Factory algorithms
\cite{Nacu2005,Mossel2005,MR2784483,Latuszynski2009,Huber2013,Huber2017,Mendo2016}
as well as extending it to quantum settings \cite{Dale2015,
  patel2019experimental, yuan2016experimental} and
specialised multivariate scenarios \cite{Huber2017,Dughmi2017}. Relations
between the Bernoulli Factory and other fundamental simulation questions in
statistics and computer science have been explored in \cite{MR3319143} and
\cite{flajolet2011buffon, Mossel2005}, respectively.  More
recently, Bernoulli Factory techniques have been successfully applied
to perform exact simulations of diffusions
\cite{Latuszynski2009,Blanchet2017}, develop perfect simulation
algorithms \cite{Blanchet2005,Flegal2012,Lee2014}, design MCMC
algorithms that can tackle  intractable likelihood models and perform
Bayesian inference
\cite{Herbei2014,Goncalves2017,Goncalves2017b,dootika2019}, design
particle filters in scenarios where weights are not available
analytically \cite{schmon2019bernoulli},
and also to reductions in mechanism design
\cite{Dughmi2017,niazadeh2017algorithms, cai2019efficient}. 

Nevertheless, designing a Bernoulli Factory algorithm for a given function $f$ is still challenging. The strategy described in \cite{Nacu2005} can be generally applied for any real analytic function $f$, but the combinatorial complexity of the implementation is prohibitive. When combined with the reverse time martingale approach of \cite{Latuszynski2009}, the implementation becomes feasible, but the running time depends on the speed of convergence of certain Bernstein polynomial envelopes and is often impractical. Specialised fast algorithms are available for linear functions \cite{Huber2013,Huber2017}, under additional assumptions on the domain, or functions admitting specific series expansions \cite{MR2958668,Latuszynski2009,Mendo2016}. However, for other classes of functions constructing a Bernoulli Factory is generally hard and even when an algorithm is available, its running time may be prohibitive. Moreover, the problem of extending the classic Bernoulli Factory setting to a multivariable one --- that is producing rolls of a die given an arbitrary other one --- has not been systematically studied.

In this paper we provide a novel constructive way to design a
Bernoulli Factory for rational functions $f$ with coefficients in
$\mathbb{R}$. Our construction can be applied to rational functions
mapping between probability simplices 
\begin{equation}
\label{f_def} 
f: \Delta^m \rightarrow \Delta^v,  \qquad m, v \geq 1, 
\end{equation} thus generalizing the classic Bernoulli Factory to a \textit{Dice Enterprise}.

Our approach relies on rephrasing the original problem as sampling
from the stationary distribution of a suitably designed Markov
chain. This is achieved by first decomposing the given rational function in a fashion inspired by \cite{Mossel2005}, and similarly based on Polya's theorem on homogeneous positive polynomials. However, the decomposition is extended in such a way that it allows to construct a Markov chain whose evolution can be
simulated by just rolling the original die. We also allow for coefficients in $\mathbb{R}$ and derive our explicit construction for multivariate scenarios. Then, the Perfect
simulations techniques, such as Coupling From The Past (CFTP)
\cite{Propp1996} or Fill's interruptible algorithm \cite{Fill1998},
can then be employed to get a sample distributed precisely as its
stationary distribution. Moreover, for $m=1$ in \eqref{f_def}, that includes the classic Bernoulli Factory
setting $m=v=1$ as a special case, a monotonic version of CFTP is proposed,
improving the efficiency of implementation. Under this scenario, we show that the method has a \q{fast simulation} (i.e., the required number of tosses has exponentially decaying tail probabilities) and the expected number of calls to the original die is linear in the degree of the resulting polynomial. To prove the result we demonstrate a fact of wider interest: the convolution of a $\text{Bin}\left(n,\frac{1}{2}\right)$ variable with any finite, integer valued random variable is log-concave when $n$ is big enough.

The paper is organised as follows:

Section \ref{sec:notation_preliminaries} introduces the notation and
notions that will be used throughout the paper. In particular, it introduces ladders, a class of
discrete probability distributions that are suitable as candidates for stationary distributions of Markov chains used in the sequel. 

Section \ref{sec:dice_enterprise} develops the Dice Enterprise for
rational functions $f: \Delta^m \rightarrow \Delta^v$, thus
generalising the usual Bernoulli Factory setting. We first show how
this problem can be rephrased as sampling from a ladder and construct
a CFTP algorithm that performs it. We prove that our proposed construction is optimal in terms of Peskun's ordering. We then analyse the efficiency of
the proposed algorithm in terms of the expected number of required
rolls of the original die and notice that it is always finite under suitable assumptions. In the \q{coin to dice} scenario (of which the Bernoulli Factory is a special case), we notice that for log-concave ladders the expected number of tosses is linear in the degree of the ladder. We then prove that it is always possible to construct such log-concave distribution.  

Section \ref{sec:examples_implementation} presents an \texttt{R} package that implements the developed method and explicative examples, validating the developed theory. In particular, we also show how a Dice Enterprise can be used to deal with $m$ independent coins and reproduce examples taken from \cite{Huber2017,Dughmi2017,Goncalves2017}. 

In Appendix \ref{sec:appendix_background} we give a background on simulating uniform random variables when the only available source of randomness is the given die, as well as a brief introduction to CFTP. Proofs of all the results are presented in Appendix \ref{sec:appendix_proofs}.

\section{Notation and Preliminaries}
\label{sec:notation_preliminaries}
Define the open $m$ dimensional probability simplex as 
\begin{equation*}
\Delta^m = \left\{\boldsymbol{p} = (p_0,\ldots,p_m) \in (0,1)^{m+1} : \sum_{i=0}^{m} p_i = 1\right\}
\end{equation*}
and by $\bar{\Delta}^m$  denote its closure. For $b \in \{0,\ldots,m\},$ by $e_b \in
\bar{\Delta}^m$ denote the $b$\textsuperscript{th} standard unit vector,
i.e.  a vector of zeros with a 1 in the
$b$\textsuperscript{th} position. We let the first element of a vector have index 0 and we interchangeably use $(p_0, p_1) \in \Delta^1$ and $p \in (0,1)$, identifying $p$ as $p_1$.
 We shall write $X \sim \boldsymbol{p}$ to denote that $X$ is a draw from the categorical distribution with parameter $\boldsymbol{p} \in \Delta^m$ on $\Omega=\{0,\ldots,m\}$.
If the vector $\boldsymbol{p}$ is not known explicitly, but there is a mechanism to sample $X \sim \boldsymbol{p}$ (e.g. via experiment or computer code), we call this mechanism a black box to sample from $\boldsymbol{p} \in \Delta^m$. Alternatively, if we want to stress that the vector $\boldsymbol{\mu} \in \Delta^n$ is given explicitly, we refer to it as known distribution $\boldsymbol{\mu}$.

In the following, we will assume that a generator of uniform random variables is available, as common in applications. However, this assumption is not restrictive from the theoretical viewpoint as we can use the given die to generate uniform random variables to any arbitrary precision. Indeed, notice that the binary representation of a uniform random variable can be seen as a sequence of fair coin tosses which we can obtain from the given die. We discuss this in more details in Appendix \ref{sec:appendix_sample_known_distributions}.

Given a rational function $f: \Delta^m \rightarrow \Delta^v$ such that $f(\boldsymbol{p})$ is a valid discrete probability distribution for all $\boldsymbol{p} \in \Delta^m$, in Section \ref{sec:dice_enterprise} we will construct a new function $\pi: \Delta^m \rightarrow \Delta^k$, named ladder, such that $\pi(\boldsymbol{p})$ is also a valid discrete probability distribution and draws from $\pi(\boldsymbol{p})$ can be transformed into draws from $f(\boldsymbol{p})$ and vice-versa. To this end, consider pairs of distributions related by disaggregation defined as follows:

\begin{defi}[Disaggregation]
\label{def:disaggregation}
Let $\boldsymbol{\mu} = (\mu_0,\ldots,\mu_k)$ and $\boldsymbol{\nu} = (\nu_0,\ldots,\nu_v)$ be probability distributions on $\Delta^k$ and $\Delta^v$ respectively with $v \leq k$. We say that $\boldsymbol{\mu}$ is a \emph{disaggregation} of $\boldsymbol{\nu}$ if there exists a partition of $\{0,\ldots,k\}$ into $v+1$ sets $A_0,A_2,\ldots,A_v$ such that 
\begin{equation*}
\nu_i = \sum_{j \in A_i} \mu_j, \qquad \text{for all } i \in \{0,\ldots,v\}.
\end{equation*}
If $\boldsymbol{\mu}$ is a disaggregation of $\boldsymbol{\nu}$, then we shall equivalently say
that $\boldsymbol{\nu}$ is an \emph{aggregation} of $\boldsymbol{\mu}$.
\end{defi}

If $\boldsymbol{\mu} = (\mu_0,\ldots,\mu_k)$ is a disaggregation of $\boldsymbol{\nu} = (\nu_0,\ldots,\nu_v)$ then, sampling from $\boldsymbol{\mu}$ is equivalent to sampling from $\boldsymbol{\nu}$ in the following sense:

\noindent Given $X \sim \boldsymbol{\mu}$, define $Y$ as $Y := i$ if $X \in A_i$. Then $Y \sim \boldsymbol{\nu}$. 

\noindent Given $X \sim \boldsymbol{\nu}$, define $Y$ by letting 
\begin{equation}
\label{eq:disaggregation}
\mathbb{P}(Y = i) = \mathbb{I}(i \in A_X)\frac{\mu_i}{\sum_{j \in A_X} \mu_j}.
\end{equation}
Then $Y \sim \boldsymbol{\mu}$.

\begin{figure}[H]
\small
%\floatbox[{\capbeside\thisfloatsetup{capbesideposition={right,top},capbesidewidth=7cm}}]{figure}[\FBwidth]
\captionsetup{width=.8\linewidth}
{\caption{ Disaggregation. A sample from $\boldsymbol{\mu}$ is directly
      mapped to a sample from $\boldsymbol{\nu}$. A sample from $\boldsymbol{\nu}$ can be mapped
      to $\boldsymbol{\mu}$ proportionally. }\label{fig:disaggregation}}
{\begin{tikzpicture}
	\edef\y{-2}
	\node at (0,0) {$\boldsymbol{\nu} \propto$};
		\node (nu_one) at (3.5,0) {$2(p_0-p_1)^2$};
		\node (nu_two) at (6,0) {$3p_0$};
		\node (nu_three) at (7.5,0) {$p_0p_1$};
	\node at (0,\y) {$\boldsymbol{\mu} \propto$};
		\node (mu_one) at (1.5,\y) {$p_0$};
		\node (mu_two) at (2.7,\y) {$p_0$};
		\node (mu_three) at (4.5,\y) {$(p_0-p_1)^2$};
		\node (mu_four) at (6.3,\y) {$p_0$};
		\node (mu_five) at (7.7,\y) {$p_0p_1$};
		\node (mu_six) at (9.7,\y) {$(p_0-p_1)^2$};
		
	\draw[latex-latex] (nu_one) -- (mu_three);
	\draw[latex-latex] (nu_one) -- (mu_six);
	\draw[latex-latex,dashed] (nu_two) -- (mu_one);
	\draw[latex-latex,dashed] (nu_two) -- (mu_two);
	\draw[latex-latex,dashed] (nu_two) -- (mu_four);
	\draw[latex-latex,dotted] (nu_three) -- (mu_five);
\end{tikzpicture}}
\normalsize
\end{figure}
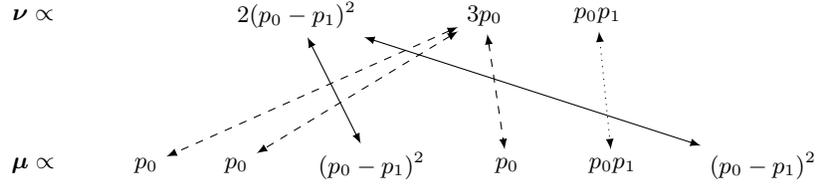

Our proposed construction requires being able to draw exactly from the stationary distribution of a Markov chain, a deed that we achieve by perfect simulation techniques. In particular, we will use Coupling From the Past (CFTP) \cite{Propp1996} a review of which is presented in Appendix \ref{sec:appendix_cftp}. 

\subsection{Ladder over $\mathbb{R}$}

We now introduce a class of probability distributions $\pi: \Delta^m
\rightarrow \Delta^k$ that will be of main interest in the remainder
of the paper. Recall that the 1-norm of a vector $\boldsymbol{a} = (a_0,\ldots,a_n)$ is $\norm{\boldsymbol{a}}_1 = \sum_{j=0}^n |a_j|$.

\begin{defi}[Multivariate ladder]
\label{def:multivariate_ladder}
For every $\boldsymbol{p} = (p_0,\ldots,p_m) \in \Delta^{m}$ let
$\pi(\boldsymbol{p}) =
(\pi_0(\boldsymbol{p}),\ldots,\pi_k(\boldsymbol{p}))$ be a probability
distribution on $\Omega = \{0,\ldots,k\}$. We say that
$\pi(\boldsymbol{p})$ is a \emph{multivariate ladder} over $\mathbb{R}$ if every $\pi_i$ is of the form
\begin{equation}
\label{eq:multivariate_ladder}
\pi_i(\boldsymbol{p}) = R_i \frac{\prod_{j=0}^{m} p_j^{n_{i,j}}}{C(\boldsymbol{p})},
\end{equation}
where
\begin{itemize}
\item $C(\boldsymbol{p})$ is a polynomial with real coefficients that
  does not admit roots in $\bar{\Delta}^m$;
% [KL: I'm still confused by this closure. Is $\{\frac{p_1}{p_1
% +(p_2-p_3)^2}, \frac{(p_2-p_3)^2}{p_1 +(p_2-p_3)^2}\}$ possible?]
% \textcolor{teal}{GM: We want this property to make sure that no
% simplifications are possible between the denominator and the
% numerator. Another way of saying this is that $C(\boldsymbol{p})$
% does not share any common root with $\prod_{j=1}^m p_j^{n_{i,j}}$
% for any $i$. However, the roots of the numerator are all of the form
% $e_b \in \bar{\Delta}^m$ so saying that $C(\boldsymbol{p})$ does not
% admit roots in $\bar{\Delta}^m$ should be the same.}
% \textcolor{blue}{But $\{\frac{p_1}{p_1 +(p_2-p_3)^2},
% \frac{(p_2-p_3)^2}{p_1 +(p_2-p_3)^2}\}$ does not have roots of the
% form $e_b$, but still seems to be ruled out by the condition, and I understood from our conversation it was not possible indeed? \textcolor{teal}{When written in ladder form, i.e. with only monomials, that function becomes $\propto \{p_1^2,p_1p_2,p1p_3,p_2^2,p_3^2,-p_2p_3\}$ and so one of the coefficients is negative. So it's not a valid ladder.}}
\item $\forall i,j$, $R_i$ is a strictly positive real constant and $n_{i,j} \in \mathbb{N}_{\geq 0}$;
\item Denote $\boldsymbol{n}_i =
  (n_{i,0},n_{i,1},\ldots,n_{i,m})$. There exists an integer $d$ such
  that $\norm{\boldsymbol{n}_i}_1 = d$ for all $i$. We will refer to $\boldsymbol{n}_i$ as the degree of $\pi_i(\boldsymbol{p})$ and to $d$ as the degree of $\pi(\boldsymbol{p})$.
\end{itemize}
Moreover, we say that $\pi(\boldsymbol{p})$ is a \textit{connected ladder} if
\begin{itemize}
\item Each $i, j \in \Omega$ are
  connected, meaning that there exists a sequence of states $(i=s_1,
  s_2, ..., s_t=j)$, such that $\norm{\boldsymbol{n}_{s_h} -
    \boldsymbol{n}_{s_{(h-1)}}}_1  \leq 2$ for all $h \in \{2,\ldots,t\}$.
%\item For each $\boldsymbol{n}_i$ there exists a $\boldsymbol{n}_j$ such that $\norm{\boldsymbol{n}_i - \boldsymbol{n}_j}_1 = 2$.
\end{itemize}
Finally, we say that $\pi(\boldsymbol{p})$ is a \textit{fine ladder} if
\begin{itemize}
\item $\boldsymbol{n}_i = \boldsymbol{n}_j$ implies $i = j$.
\end{itemize}
\end{defi}

Figure \ref{fig:multivariate_ladder} gives a graphical representation of a multivariate ladder and motivates the following concept of neighbourhood.

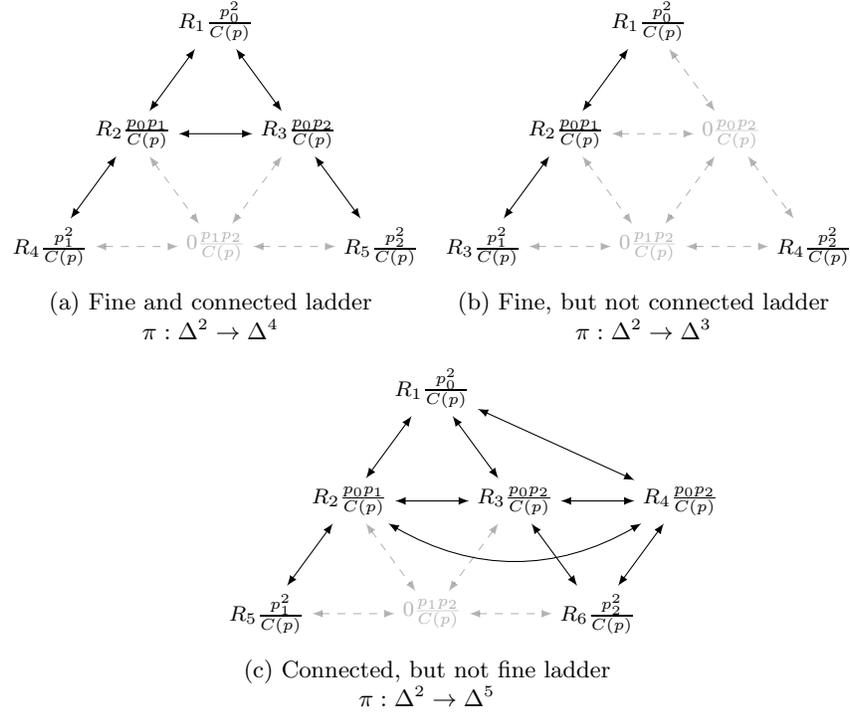
\begin{figure}[H]
    \centering
    \begin{subfigure}[b]{0.45\textwidth}
\begin{tikzpicture}
%\node (p_\n_\k) at (2.2*\k-2.2*\n/2,-\n*1.5) {$R_{\counter}\frac{p_3^{\rw}}{C(p)}$};
\edef\rows{2}
\node (p_0_0) at (2.2*0-2.2*0/2,-0*1.5) {$R_1\frac{p_0^2}{C(p)}$};
\node (p_1_0) at (2.2*0-2.2*1/2,-1*1.5) {$R_2\frac{p_0p_1}{C(p)}$};
\node (p_1_1) at (2.2*1-2.2*1/2,-1*1.5) {$R_3\frac{p_0p_2}{C(p)}$};
\node (p_2_0) at (2.2*0-2.2*2/2,-2*1.5) {$R_4\frac{p_1^2}{C(p)}$};
\node[opacity=.3] (p_2_1) at (2.2*1-2.2*2/2,-2*1.5) {$0\frac{p_1p_2}{C(p)}$};
\node (p_2_2) at (2.2*2-2.2*2/2,-2*1.5) {$R_5\frac{p_2^2}{C(p)}$};

\draw[latex-latex] (p_0_0) -- (p_1_0);
\draw[latex-latex] (p_0_0) -- (p_1_1);
\draw[latex-latex] (p_1_0) -- (p_1_1);
\draw[latex-latex] (p_1_0) -- (p_2_0);
\draw[latex-latex, opacity=.3, dashed] (p_1_0) -- (p_2_1);
\draw[latex-latex, opacity=.3, dashed] (p_1_1) -- (p_2_1);
\draw[latex-latex] (p_1_1) -- (p_2_2);
\draw[latex-latex, opacity=.3, dashed] (p_2_0) -- (p_2_1);
\draw[latex-latex, opacity=.3, dashed] (p_2_1) -- (p_2_2);

\end{tikzpicture}
		\captionsetup{justification=centering}
        \caption{Fine and connected ladder \linebreak $\pi: \Delta^2 \rightarrow \Delta^4$}
        \label{fig:multivariate_ladder_fine_connected}
    \end{subfigure}
    ~ %add desired spacing between images, e. g. ~, \quad, \qquad, \hfill etc. 
    \begin{subfigure}[b]{0.45\textwidth}
     \begin{tikzpicture}
%\node (p_\n_\k) at (2.2*\k-2.2*\n/2,-\n*1.5) {$R_{\counter}\frac{p_3^{\rw}}{C(p)}$};
\edef\rows{2}
\node (p_0_0) at (2.2*0-2.2*0/2,-0*1.5) {$R_1\frac{p_0^2}{C(p)}$};
\node (p_1_0) at (2.2*0-2.2*1/2,-1*1.5) {$R_2\frac{p_0p_1}{C(p)}$};
\node[opacity=.3] (p_1_1) at (2.2*1-2.2*1/2,-1*1.5) {$0\frac{p_0p_2}{C(p)}$};
\node (p_2_0) at (2.2*0-2.2*2/2,-2*1.5) {$R_3\frac{p_1^2}{C(p)}$};
\node[opacity=.3] (p_2_1) at (2.2*1-2.2*2/2,-2*1.5) {$0\frac{p_1p_2}{C(p)}$};
\node (p_2_2) at (2.2*2-2.2*2/2,-2*1.5) {$R_4\frac{p_2^2}{C(p)}$};

\draw[latex-latex] (p_0_0) -- (p_1_0);
\draw[latex-latex, opacity=.3, dashed] (p_0_0) -- (p_1_1);
\draw[latex-latex, opacity=.3, dashed] (p_1_0) -- (p_1_1);
\draw[latex-latex] (p_1_0) -- (p_2_0);
\draw[latex-latex, opacity=.3, dashed] (p_1_0) -- (p_2_1);
\draw[latex-latex, opacity=.3, dashed] (p_1_1) -- (p_2_1);
\draw[latex-latex, opacity=.3, dashed] (p_1_1) -- (p_2_2);
\draw[latex-latex, opacity=.3, dashed] (p_2_0) -- (p_2_1);
\draw[latex-latex, opacity=.3, dashed] (p_2_1) -- (p_2_2);

\end{tikzpicture}
		\captionsetup{justification=centering}
        \caption{Fine, but not connected ladder \linebreak $\pi:\Delta^2 \rightarrow \Delta^3$}
        \label{fig:multivariate_ladder_fine}
    \end{subfigure}
    ~ %add desired spacing between images, e. g. ~, \quad, \qquad, \hfill etc. 
    \begin{subfigure}[b]{0.45\textwidth}
      \begin{tikzpicture}
%\node (p_\n_\k) at (2.2*\k-2.2*\n/2,-\n*1.5) {$R_{\counter}\frac{p_3^{\rw}}{C(p)}$};
\edef\rows{2}
\node (p_0_0) at (2.2*0-2.2*0/2,-0*1.5) {$R_1\frac{p_0^2}{C(p)}$};
\node (p_1_0) at (2.2*0-2.2*1/2,-1*1.5) {$R_2\frac{p_0p_1}{C(p)}$};
\node (p_1_1) at (2.2*1-2.2*1/2,-1*1.5) {$R_3\frac{p_0p_2}{C(p)}$};
\node (p_extra) at (2.2*2-2.2*1/2,-1*1.5) {$R_4\frac{p_0p_2}{C(p)}$};
\node (p_2_0) at (2.2*0-2.2*2/2,-2*1.5) {$R_5\frac{p_1^2}{C(p)}$};
\node[opacity=.3] (p_2_1) at (2.2*1-2.2*2/2,-2*1.5) {$0\frac{p_1p_2}{C(p)}$};
\node (p_2_2) at (2.2*2-2.2*2/2,-2*1.5) {$R_6\frac{p_2^2}{C(p)}$};

\draw[latex-latex] (p_0_0) -- (p_1_0);
\draw[latex-latex] (p_0_0) -- (p_1_1);
\draw[latex-latex] (p_1_0) -- (p_1_1);
\draw[latex-latex] (p_1_0) -- (p_2_0);
\draw[latex-latex, opacity=.3, dashed] (p_1_0) -- (p_2_1);
\draw[latex-latex, opacity=.3, dashed] (p_1_1) -- (p_2_1);
\draw[latex-latex] (p_1_1) -- (p_2_2);
\draw[latex-latex, opacity=.3, dashed] (p_2_0) -- (p_2_1);
\draw[latex-latex, opacity=.3, dashed] (p_2_1) -- (p_2_2);
\draw[latex-latex] (p_extra) -- (p_0_0);
\draw[latex-latex] (p_extra) -- (p_2_2);
\draw[latex-latex, bend left] (p_extra) to (p_1_0);
\draw[latex-latex] (p_extra) -- (p_1_1);

\end{tikzpicture}
		\captionsetup{justification=centering}
        \caption{Connected, but not fine ladder \linebreak $\pi: \Delta^2 \rightarrow \Delta^5$}
        \label{fig:multivariate_ladder_connected}
    \end{subfigure}
    \caption{Multivariate ladders over $\mathbb{R}$. Edges represent connected states.}\label{fig:multivariate_ladder}
\end{figure}

\begin{defi}[Neighbourhoods on ladders]
\label{def:neighbourhood_ladder}
On a multivariate ladder $\pi: \Delta^{m} \rightarrow \Delta^{k}$
define the \emph{neighbourhood of $i \in \Omega$} as
$\mathcal{N}(i) = \{j \in \Omega \setminus \{i\}: \norm{\boldsymbol{n}_i - \boldsymbol{n}_j} \leq
2\}$. Note that for connected ladders $\mathcal{N}(i)$ must
have at least one element for each $i$ (in the non-trivial case of $k>1$). 
\end{defi}

\subsubsection{Operations on ladders}
%\textcolor{teal}{GM: This subsection is new. I think it is a useful section, as it makes clearer what we mean with \q{augmenting the ladder} in the following sections, especially in the Efficiency section. Also it makes some proofs less redundant.}

We now introduce three operations on ladders of which we will make
extensive use: \textit{increasing the degree} of
a ladder, \textit{thinning} and \textit{augmenting} a ladder.

\begin{defi}[Increasing the degree]
\label{prop:increasing_degree_ladder}
Let $\pi: \Delta^m \rightarrow \Delta^{k}$ be a multivariate ladder of
degree $d$. \emph{Increasing the degree of $\pi$} yields a new
ladder $\pi': \Delta^m \rightarrow \Delta^{(k+1)(m+1)-1}$ of degree $(d+1)$ with
probabilities $\pi'_l(\boldsymbol{p})$ on $\Omega' = \{0,\ldots,(k+1)(m+1)-1\}$
of the form $\pi'_l(\boldsymbol{p}) := \pi_{i}(\boldsymbol{p})p_{j}$, where  $i \in \{0, \ldots, k\}$ and $j \in \{0,\ldots,m\}$ are the unique solution of $l = i(m+1)+j$.  \end{defi}

Increasing the degree corresponds to multiplying each state by $p_0,\ldots,p_m$ and the resulting ladder $\pi'(\boldsymbol{p})$ is a disaggregation of $\pi(\boldsymbol{p})$. Indeed, let  $A_0,\ldots,A_k$ be 
\begin{equation*}
A_i: = \{i(m+1), \ldots, i(m+1)+ m\}.
\end{equation*}
Then  definition \ref{def:disaggregation} is satisfied, since
\begin{equation*}
\sum_{a \in A_i} \pi'_a(\boldsymbol{p}) = \sum_{j=0}^{m} \pi_i(\boldsymbol{p})p_j = \pi_i(\boldsymbol{p})\sum_{j=0}^{m} p_j = \pi_i(\boldsymbol{p}).
\end{equation*}

\begin{defi}[Thinning]
\label{prop:thinning_ladder}
Let $\pi: \Delta^m \rightarrow \Delta^k$ be a multivariate ladder of degree $d$. \emph{Thinning} $\pi$ yields a fine ladder $\pi'$ by joining all the states of $\pi$ with the same monomial. Thus $\pi': \Delta^m \rightarrow \Delta^w$ where $k \geq w := | \{\boldsymbol{n}_0, \ldots, \boldsymbol{n}_{k} \}| $, and by \eqref{eq:multivariate_ladder} the probabilities $\pi'_l(\boldsymbol{p})$ on $\Omega' = \{0,\ldots,w\}$ are
of the form $\pi'_l(\boldsymbol{p}) := \frac{R'_l \boldsymbol{p}^{\boldsymbol{n}'_l}}{C(\boldsymbol{p})}$ where 
$R'_l = \sum_{i: \boldsymbol{n}_i = \boldsymbol{n}'_l} R_i.$
\end{defi}

Clearly, $\pi(\boldsymbol{p})$ is a disaggregation of the resulting
$\pi'(\boldsymbol{p})$. Moreover, if $\pi$ is a connected ladder, then
so is $\pi'$.

Increasing the degree will typically not result in a fine ladder as it produces \q{redundant} states, however thinning can be applied subsequently. We will refer to increasing the degree of the ladder first and thinning it afterwards, as \textit{augmenting} the ladder. 

\begin{defi}[Augmenting]
\label{def:augmenting_ladder}
Let $\pi: \Delta^m \rightarrow \Delta^k$ be a multivariate ladder of degree $d$. The augmented ladder $\pi': \Delta^m \rightarrow \Delta^w$, where $w < \min\{(k+1)(m+1), {{d+m+1}\choose{m}}\},$ is obtained by first increasing the degree of $\pi$ and then thinning it. 
\end{defi}

The fact that $w < {{d+m+1}\choose{m}}$ in the augmented ladder of degree $d+1$, follows by noticing that there are at most ${{d+m-1}\choose{m-1}}$ homogeneous monomials of degree $d$ in $m$ variables. Importantly, sampling from $\pi$ and its augmented ladder $\pi'$ is equivalent, since it is enough to transform the sample in line with the disaggregation and aggregation steps applied. Finally we make the following important remark that connects the operation of augmenting a ladder to that of convolution of random variables. 

\begin{remark}
\label{rmk:augmenting_convolution}
Let $m=1$, $\pi$ be a fine ladder and assume $\boldsymbol{n}_i$'s are ordered lexicographically. Moreover, let $W \sim \text{Ber}(p)$ be independent of $ X \sim \pi(p)$ and $Y$ be the convolution of the two, that is $Y = X+W$. Then, $Y \sim \pi'(p)$, where $\pi'$ is the augmented $\pi$.
\end{remark}

Notice that given a multivariate ladder $\pi: \Delta^m \rightarrow \Delta^k$, augmenting it enough times yields a fine and connected ladder.	

\begin{prop}
\label{prop:multivariate_impose_connected_fineness_condition}
Let $\pi: \Delta^m \rightarrow \Delta^k$ be a multivariate
ladder of degree~$d$. Augment $\pi$ $d$ times to construct $\pi':
\Delta^m  \rightarrow \Delta^w$, where $w < \min\{(k+1)(m+1)^d, {{2d+m}\choose{m}}\}$. Then $\pi'$ is a fine and connected ladder and sampling from $\pi'$ is equivalent to sampling from $\pi$. 
\end{prop}

In practice, it may be enough to augment the ladder $\pi$ less than $d$ times to produce a fine and connected ladder $\pi'$.

\subsubsection{Univariate fine and connected ladder over $\mathbb{R}$}\label{def:univariate_connected_fine_ladder}

Consider the special case of fine and connected ladders where $m=1$, which will be of particular interest for the monotone CFTP implementation. We shall call such ladders univariate and denote $p_0 = 1-p$ and $p_1 = p$.  The condition that $C(p)$ has no roots in $[0,1]$ implies $d=k$ and the probabilities take the form of
\begin{equation}
\label{eq:univariate_connected_fine_ladder}
\pi_i(p) = R_i \frac{p^{i}(1-p)^{k-i}}{C(p)}, \qquad i = 0, \ldots, k.
\end{equation}
This case corresponds to having access to a $p$-coin and simulating a $(k+1)$-sided die, so that the classic Bernoulli Factory setting falls in this scenario.

\section{A dice enterprise for rational functions}
\label{sec:dice_enterprise}

We now tackle the problem of designing an algorithm that given a die where the probability $\boldsymbol{p} \in \Delta^m$ of rolling each face is unknown, produces rolls of a  die (possibly having a different number of faces $v$) where the probability associated to each face is $f(p)$, i.e. given by a rational function $f: \Delta^m \rightarrow \Delta^v$. We first show how to decompose a rational function $f$ into a fine and connected ladder $\pi: \Delta^m \rightarrow \Delta^k$ such that sampling from $f(\boldsymbol{p})$ is equivalent to sampling from $\pi(\boldsymbol{p})$. Then, we detail how to construct a Markov chain that admits such $\pi(\boldsymbol{p})$ as its stationary distribution. Finally, we apply CFTP perfect sampling technique to get a sample exactly from $\pi(\boldsymbol{p})$. The algorithm produces exact draws from $\pi(\boldsymbol{p})$ for any $m$, but for the case $m = 1$ the CFTP has a monotonic implementation with improved efficiency. 

\subsection{Construction of $\pi$}

Theorem \ref{thm:decomposition_multivariate} below shows that for any rational
function $f: \Delta^m \rightarrow \Delta^v$, a fine and connected ladder $\pi: \Delta^m \rightarrow \Delta^k$ can be constructed, such that sampling from $f$ is equivalent to sampling from $\pi$. The construction is explicit and among others, builds on the ideas of \cite{Mossel2005}. 

Roughly speaking the key steps of our proposed method are the following: 

\begin{algorithmic}[1]
\State{Let $f(\boldsymbol{p}) = (f_0(\boldsymbol{p}),\ldots,f_v(\boldsymbol{p})) = \left(\frac{D_0(\boldsymbol{p})}{E_0(\boldsymbol{p})},\ldots,\frac{D_v(\boldsymbol{p})}{E_v(\boldsymbol{p})}\right)$ be a given rational function where $D_i(\boldsymbol{p})$ and $E_i(\boldsymbol{p})$ are positive and relatively prime polynomials.}
\State{Apply Lemma \ref{lemma:decomposition} (presented in the appendix) to each $f_i(\boldsymbol{p})$, so that $f(\boldsymbol{p}) = \left(\frac{d_0(\boldsymbol{p})}{e_0(\boldsymbol{p})},\ldots,\frac{d_v(\boldsymbol{p})}{e_v(\boldsymbol{p})}\right)$ and each $d_i(\boldsymbol{p})$ and $e_i(\boldsymbol{p})$ is an homogeneous polynomial with positive coefficients.} 
\State{Rewrite $f(\boldsymbol{p})$ using a common denominator, so that \linebreak $f(\boldsymbol{p}) =
  \frac{1}{C(\boldsymbol{p})}(G_0(\boldsymbol{p}),\ldots,G_v(\boldsymbol{p}))$.}
\State{Rewrite each polynomial $G_i(\boldsymbol{p})$ as a homogeneous
  polynomial of degree~$d$. }
\State{Using Proposition \ref{prop:multivariate_impose_connected_fineness_condition}, construct a fine and connected ladder $\pi(\boldsymbol{p})$ sampling from which is equivalent to sampling from $f(\boldsymbol{p})$.}
\end{algorithmic}

Detailed construction can be found in the proof of the following theorem and is also illustrated in Example \ref{ex:toy_bf}.

\begin{thm}
\label{thm:decomposition_multivariate}
Let $f: \Delta^m \rightarrow \Delta^v$ be a probability distribution
such that every $f_i(\boldsymbol{p})$ is a rational function with real
coefficients. Then, one can explicitly construct a fine and connected
ladder $\pi: \Delta^m \rightarrow \Delta^k$ such that sampling from
$\pi$ is equivalent to sampling from $f$.
\end{thm}
Notice that we assume that $f(\boldsymbol{p}) \in \Delta^v$ for every $\boldsymbol{p} \in \Delta^m$. This rules out functions such as $f(p) = \min(2p,1)$ - as expected since a Bernoulli Factory for such function  cannot be constructed \cite{Keane1994}.

\subsection{Construction of the Markov chain}

Let $\pi: \Delta^m \rightarrow \Delta^k$ be a fine and connected
ladder. We now consider the problem of designing a Markov chain that admits it as its stationary distribution. The main idea behind the construction is depicted in Figure \ref{fig:multivariate_MC}: a roll of the die determines the
possible directions for the next move. We then draw a Uniform r.v. and decide whether the chain stays still or moves to a specific state. We can then write the off diagonal entries of the transition matrix $P$ of the chain as an entrywise product of $V$ and $W$, i.e.
\begin{equation}
\label{eq:transition_matrix_product}
P_{i,j} = \begin{cases}
V_{i,j} \cdot W_{i,j} &\quad \text{if } i \neq j \\
1-\sum_{h\neq i} P_{i,h} &\quad \text{if } i = j
\end{cases}
\end{equation}
where $V$ is a matrix of real numbers in $[0,1]$ and $W$ is a matrix where the off diagonal elements are either null or equal to $p_b$ for some $b \in \{0,\ldots,m\}$.

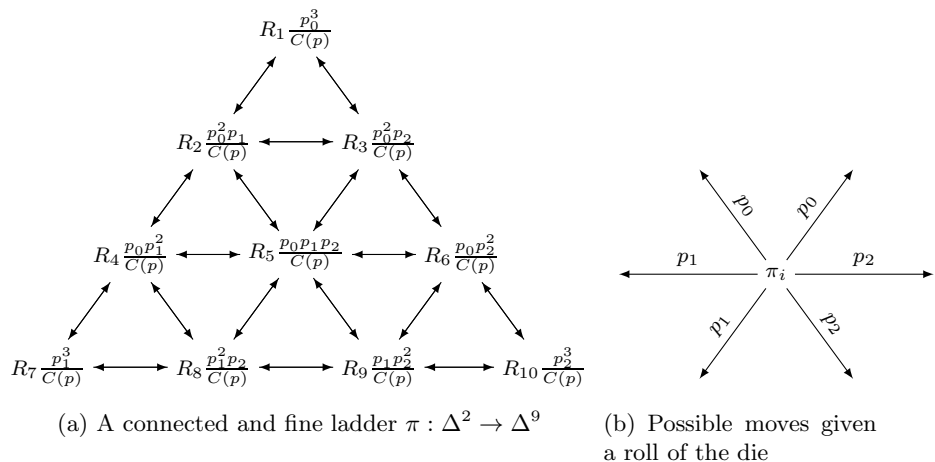
\begin{figure}[t]
    \centering
    \begin{subfigure}[t]{0.65\textwidth}
\begin{tikzpicture}
	%Level 0
%	\draw (0,0) node[] {$R_7\frac{p_2^3}{C(p)}$};
%	\draw (3.5,0)  node[] {$R_8\frac{p_2^2p_3}{C(p)}$};
%	\draw (7,0)  node[] {$R_9\frac{p_2p_3^2}{C(p)}$};
%	\draw (10.5,0) node[] {$R_{10}\frac{p_3^3}{C(p)}$};
%	
%	%Level 1
%	\draw (1.75,1.5) node[] {$R_4\frac{p_1p_2^2}{C(p)}$};

\edef\rows{3} %How many levels of the triangle!

\foreach \n in {0,...,\rows} {
  \foreach \k in {0,...,\n} {
   \pgfmathtruncatemacro{\p}{\rows-\n}
   \pgfmathtruncatemacro{\q}{\n-\k}
   \pgfmathtruncatemacro{\r}{\k}
   \pgfmathtruncatemacro{\counter}{(\n)*(\n+1)/2+\k+1} %Triangle number + \k
   
   %Make exponent =1 not appear
   \ifnum\p=1
   		\def\pw{}
   \else
   		\def\pw{\p}
   \fi
   \ifnum\q=1
   		\def\qw{}
   \else
   		\def\qw{\q}
   \fi
   \ifnum\r=1
   		\def\rw{}
   \else
   		\def\rw{\r}
   \fi
   
   %Make exponent = 0 not appear
	\ifnum\p=0
		\ifnum\q=0
			\node (p_\n_\k) at (2.2*\k-2.2*\n/2,-\n*1.5) {$R_{\counter}\frac{p_2^{\rw}}{C(p)}$};
		\else
			\ifnum\r=0
				\node (p_\n_\k) at (2.2*\k-2.2*\n/2,-\n*1.5) {$R_{\counter}\frac{p_1^{\qw}}{C(p)}$};
			\else
				\node (p_\n_\k) at (2.2*\k-2.2*\n/2,-\n*1.5) {$R_{\counter}\frac{p_1^{\qw}p_2^{\rw}}{C(p)}$};
			\fi
		\fi
	\else
		%p>0
		\ifnum\q=0
			\ifnum\r=0
				\node (p_\n_\k) at (2.2*\k-2.2*\n/2,-\n*1.5) {$R_{\counter}\frac{p_0^{\pw}}{C(p)}$};
			\else
				\node (p_\n_\k) at (2.2*\k-2.2*\n/2,-\n*1.5) {$R_{\counter}\frac{p_0^{\pw}p_2^{\rw}}{C(p)}$};
			\fi
		\else
			\ifnum\r=0
				\node (p_\n_\k) at (2.2*\k-2.2*\n/2,-\n*1.5) {$R_{\counter}\frac{p_0^{\pw}p_1^{\qw}}{C(p)}$};
			\else
				\node (p_\n_\k) at (2.2*\k-2.2*\n/2,-\n*1.5) {$R_{\counter}\frac{p_0^{\pw}p_1^{\qw}p_2^{\rw}}{C(p)}$};
			\fi
		\fi
	\fi
   
  }
}

\foreach \n in {0,...,\rows} {
  \foreach \k in {0,...,\n} {
    \pgfmathtruncatemacro{\nm}{\n-1}
    \pgfmathtruncatemacro{\np}{\n+1}
    \pgfmathtruncatemacro{\km}{\k-1}
    \pgfmathtruncatemacro{\kp}{\k+1}
    \pgfmathtruncatemacro{\rowsp}{\rows+1}
	%Connect to the ones above
	\ifnum\nm>-1
	\ifnum\nm>\km
		\draw[-latex] (p_\n_\k) -- (p_\nm_\k);
	\fi
	\fi
	\ifnum\nm>-1
	\ifnum\km>-1
		\draw[-latex] (p_\n_\k) -- (p_\nm_\km);
	\fi
	\fi
	
	%Connect to the ones below
	\ifnum\np<\rowsp
		\draw[-latex] (p_\n_\k) -- (p_\np_\k);
		\draw[-latex] (p_\n_\k) -- (p_\np_\kp);
	\fi
	
	%Connect to the same level
	\ifnum\kp<\np
		\draw[-latex] (p_\n_\k) -- (p_\n_\kp);
	\fi
	\ifnum\km>-1
		\draw[-latex] (p_\n_\k) -- (p_\n_\km);
	\fi

  }
}

\end{tikzpicture}
        \caption{A connected and fine ladder $\pi: \Delta^2 \rightarrow \Delta^{9}$}
        \label{fig:multivariate_MC_4}
    \end{subfigure}
    \begin{subfigure}[t]{0.3\textwidth}
        \begin{tikzpicture}
%\node (p_\n_\k) at (2.2*\k-2.2*\n/2,-\n*1.5) {$R_{\counter}\frac{p_3^{\rw}}{C(p)}$};
\edef\rows{2}
\node (p_1_0) at (2.2*0-2.2*1/2,-1*1.5) {};
\node (p_1_1) at (2.2*1-2.2*1/2,-1*1.5) {};
\node (p_2_0) at (2.2*0-2.2*2/2,-2*1.5) {};
\node (p_2_1) at (2.2*1-2.2*2/2,-2*1.5) {$\pi_i$};
\node (p_2_2) at (2.2*2-2.2*2/2,-2*1.5) {};
\node (p_3_1) at (2.2*1-2.2*3/2,-3*1.5) {};
\node (p_3_2) at (2.2*2-2.2*3/2,-3*1.5) {};

\draw[latex-] (p_1_0) to node[sloped, anchor=center, above]{$p_0$}(p_2_1);
\draw[latex-] (p_1_1) to node[sloped, anchor=center, above]{$p_0$}(p_2_1);
\draw[latex-] (p_2_0) to node[sloped, anchor=center, above]{$p_1$}(p_2_1);
\draw[-latex] (p_2_1) to node[sloped, anchor=center, above]{$p_2$}(p_2_2);
\draw[latex-] (p_3_1) to node[sloped, anchor=center, above]{$p_1$}(p_2_1);
\draw[latex-] (p_3_2) to node[sloped, anchor=center, above]{$p_2$}(p_2_1);

\end{tikzpicture}
        \caption{Possible moves given a roll of the die}
        \label{fig:multivariate_MC_moves}
    \end{subfigure}
    \caption{Markov chain structure for a multivariate ladder}\label{fig:multivariate_MC}
\end{figure}

It is convenient to introduce the following definition of neighbourhood that, once the $(m+1)$-sided die has been rolled, details where the chain may move.

\begin{defi}
Let $\pi: \Delta^{m} \rightarrow \Delta^{k}$ be a fine and connected
multivariate ladder. Let $b \in \{0,\ldots,m\}$ and let $\boldsymbol{e}_b \in
\bar{\Delta}^m$ be the $b$\textsuperscript{th} standard unit vector. For each $i \in \Omega = \{0,\ldots,k\}$
define the neighbourhood of $i$ in the direction of $b$ as 
\begin{equation}
\label{eq:neighbourhood_rolled}
\mathcal{N}_b(i) = \{j \in \Omega \setminus \{i\}: \norm{\boldsymbol{n}_i-\boldsymbol{n}_j+\boldsymbol{e}_b}_1 = 1\}.
\end{equation}
We will also denote 
\begin{equation}
\label{eq:neighbourhood_coefficients}
\mathcal{S}_b(i) = \sum_{j \in \mathcal{N}_b(i)} R_j.
\end{equation}
\end{defi}

Unlike $\mathcal{N}(i)$ (cf. Definition \ref{def:neighbourhood_ladder}), $\mathcal{N}_b(i)$ may be empty and 
\begin{equation*}
\mathcal{N}(i) = \bigcup_{b \in \{0,\ldots,m\}} \mathcal{N}_b(i).
\end{equation*}

We can now set the %off-diagonal 
elements of the matrix $W$ in \eqref{eq:transition_matrix_product} as
\begin{equation}
\label{eq:transition_matrix_W}
W_{i,j} = \begin{cases}
p_b &\quad \text{if } j \in \mathcal{N}_b(i) \\
0 &\quad \text{if } j \not\in \mathcal{N}(i).
\end{cases}
\end{equation}

We are left to define the matrix $V$ in eq. \eqref{eq:transition_matrix_product}. In principle, to speed up the convergence of CFTP it is good practice to reduce the
mixing time of the chain \cite{Propp1996}. A related and more
operational criterion is that of Peskun ordering \cite{Peskun1973}:

\begin{defi}
Given two reversible Markov chains with the same stationary distribution $\pi$ and with transition matrices $P$ and $Q$, we say that $Q$ dominates $P$ in Peskun sense, and write $Q \succeq_P P,$ if each of the off diagonal elements of $Q$ are greater or equal to the corresponding elements of $P$.
\end{defi}

If  $Q \succeq_P P$ then, for any $\pi$ integrable target
function $f,$ using $Q$ results in a smaller asymptotic
variance in the Markov chain CLT than using $P$. Moreover, for positive operators, if $Q \succeq_P P$ then the $Q$-chain converges in total variation more rapidly towards the stationary distribution \cite{Mira2001}. Consequently, Peskun ordering represents a valuable tool to assess our choice of the transition matrix.

The requirement of having a reversible Markov chain leads to a natural choice for the off-diagonal entries of the matrix $V$:
\begin{equation}
\label{eq:transition_probs_ladder_multivariate_suboptimal}
V_{i,j} =
\dfrac{R_j}{\mathcal{S}_b(i) \vee \mathcal{S}_c(j)} \qquad \text{if } j \in \mathcal{N}_b(i) \text{ and } i \in \mathcal{N}_c(j).
\end{equation}
It is easy to check that this choice produces a transition matrix $P$ that satisfies the detailed balance condition and that $\pi(\boldsymbol{p})$ is the stationary distribution of the chain. However, this choice may not be optimal in Peskun's ordering. Therefore, we propose a different construction and define the matrix $V$ of eq. n \eqref{eq:transition_matrix_product} iteratively. First, we select the state $i \in \Omega$ and the roll $b \in \{0,\ldots,m\}$ that maximises $\mathcal{S}_b(i)$; i.e. we identify the pair of state and die face from which it is easiest for the chain to move away from.  In particular,  for each $j \in \mathcal{N}_b(i)$ we assign $V_{i,j} = \frac{R_j}{\mathcal{S}_b(i)}$. Since the construction shall yield a reversible Markov chain, we also set $V_{j,i} = \frac{R_i}{\mathcal{S}_b(i)}$ (notice that the denominator is purposely set equal to $\mathcal{S}_b(i)$ so that detailed balance condition is satisfied). We then proceed analogously, now identifying a new pair of state and die face for which the probability of moving out from such state (going in the direction of the given face) is maximised, taking into account the entries of $V$ that have already been fixed. The detailed procedure is described in Algorithm \ref{alg:construction_markov_chain}. We prove in Proposition \ref{prop:multivariate_markov_chain} that since at each step we ensure that the detailed balance condition holds, this leads to a valid reversible chain and $\pi(\boldsymbol{p})$ is its unique stationary distribution. Importantly, the so-constructed transition matrix $P$ is optimal in Peskun ordering. We prove this in Proposition \ref{prop:Peskun_optimal}, the key point being that since at each step we select the pair of state and die face that maximises the off-diagonal elements of $V$, any other choice would produce a matrix $P$ that is no longer a valid stochastic matrix for all $\boldsymbol{p} \in \Delta^m$. 

\begin{algorithm}[t]
	\caption{Construction of the Markov chain transition matrix}
	\label{alg:construction_markov_chain}
	\hspace*{\algorithmicindent}\justifying\textbf{Input}: A multivariate ladder $\pi: \Delta^m \rightarrow \Delta^k$ on $\Omega = \{0,\ldots,k\}$
	\\
	\hspace*{\algorithmicindent}\justifying\textbf{Output}: The matrix $V$ composing the transition kernel in equation \eqref{eq:transition_matrix_product}.
	\begin{algorithmic}[1]
		\Statex{\textit{Initialisation step}}
		\State{Initialise $V$ as a $(k+1)\times (k+1)$ null matrix}
		\State{For all $i \in \Omega$ and $b \in \{0,\ldots,m\}$ set $\mathcal{N}_b(i)$ as in eq. \eqref{eq:neighbourhood_rolled}, $\mathcal{S}_b(i)$ as in eq. \eqref{eq:neighbourhood_coefficients}, $\mathcal{W}_b(i) \gets 0$}
		\Statex{\textit{Main loop}}
		\Repeat
		\State{Set $b, i \gets \argmax_{b,i} \mathcal{S}_b(i)$} \label{alg:line_max_arg}
		\ForEach{$j \in \mathcal{N}_b(i)$}
			\Statex{\textit{Assign maximum probability of moving from state $i$ having rolled $b$}}
			\State{$V_{i,j} \gets R_j/S_b(i)$, $\mathcal{N}_b(i) \gets \mathcal{N}_b(i) \setminus \{j\}$, $\mathcal{W}_b(i) \gets \mathcal{W}_b(i)+R_j/S_b(i)$} \label{alg:line:max}
			%\State{$V_{i,j} \gets R_j/S_b(i)$} \label{alg:line:max}
			\Statex{\textit{Assign values for the reverse move}}
			\State{Set $c$ such that $i \in \mathcal{N}_c(j)$}
			\State{$V_{j,i} \gets R_i/S_b(i)$, $\mathcal{N}_c(j) \gets \mathcal{N}_c(j) \setminus \{i\}$, $\mathcal{W}_c(j) \gets \mathcal{W}_c(j)+R_i/S_b(i)$} \label{alg:line:reverse}
			\Statex{\textit{Update $S_c(j)$ to take into consideration the new value of $V_{j,i}$}}
			\State{$S_c(j) \gets \left( \sum_{h \in \mathcal{N}_c(j)} R_h \right)\left(1-\mathcal{W}_c(j)\right)^{-1}$}
		\EndFor
		\Statex{\textit{Update $\mathcal{S}_b(i)$}}
		\State{$\mathcal{S}_b(i) \gets 0$}
		\Until{$\mathcal{N}_b(i) = \emptyset, \forall i \in \Omega, b \in \{0,\ldots,m\}$}
	\end{algorithmic}
\end{algorithm}

\begin{prop}
\label{prop:multivariate_markov_chain}
Let $\pi: \Delta^m \rightarrow \Delta^k$ be a fine and connected
ladder. Consider a discrete-time Markov chain $(X_t)_{t \in
  \mathbb{N}}$ on $\Omega = \{0,\ldots,k\}$. Let $P$ as in \eqref{eq:transition_matrix_product} be the transition matrix of the chain, where $W$ is as in \eqref{eq:transition_matrix_W} and $V$ is the matrix output by Algorithm \ref{alg:construction_markov_chain}. Then, $(X_t)_{t \in \mathbb{N}}$ is a
time-reversible Markov chain that admits $\pi(\boldsymbol{p})$ as its unique stationary distribution.
\end{prop}

\begin{prop}
\label{prop:Peskun_optimal} 
Let $\pi: \Delta^m \rightarrow \Delta^k$ be a fine and connected
ladder. Consider the Markov chain defined in Proposition
\ref{prop:multivariate_markov_chain}. Then, there does not exist a
reversible Markov chain with the same adjacency structure and
stationary distribution that dominates it in the Peskun sense.
\end{prop}

\begin{ex}
\label{ex:multivariate}
Consider the multivariate ladder 
\begin{equation*}
\pi(p_0,p_1,p_2) \propto \left(\underbrace{\sqrt{2}p_0^3}_{\pi_0},\underbrace{p_0^2p_2}_{\pi_1},\underbrace{\frac{1}{4}p_0p_1^2}_{\pi_2},\underbrace{2p_0p_1p_2}_{\pi_3},\underbrace{\frac{1}{2}p_0p_2^2}_{\pi_4},\underbrace{\frac{3}{4}p_1^2p_2}_{\pi_5}\right).
\end{equation*}
We can graphically represent the ladder as

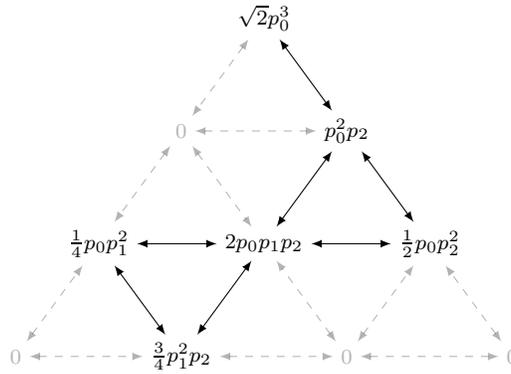
\begin{figure}[H]
\centering
\begin{tikzpicture}
%\node (p_\n_\k) at (2.2*\k-2.2*\n/2,-\n*1.5) {$R_{\counter}\frac{p_3^{\rw}}{C(p)}$};
\edef\rows{2}
\node (p_0_0) at (2.2*0-2.2*0/2,-0*1.5) {$\sqrt{2}p_0^3$};
\node[opacity=.3] (p_1_0) at (2.2*0-2.2*1/2,-1*1.5) {$0$};
\node (p_1_1) at (2.2*1-2.2*1/2,-1*1.5) {$p_0^2p_2$};
\node (p_2_0) at (2.2*0-2.2*2/2,-2*1.5) {$\frac{1}{4}p_0p_1^2$};
\node (p_2_1) at (2.2*1-2.2*2/2,-2*1.5) {$2p_0p_1p_2$};
\node (p_2_2) at (2.2*2-2.2*2/2,-2*1.5) {$\frac{1}{2}p_0p_2^2$};
\node[opacity=.3] (p_3_0) at (2.2*0-2.2*3/2,-3*1.5) {$0$};
\node (p_3_1) at (2.2*1-2.2*3/2,-3*1.5) {$\frac{3}{4}p_1^2p_2$};
\node[opacity=.3] (p_3_2) at (2.2*2-2.2*3/2,-3*1.5) {$0$};
\node[opacity=.3] (p_3_3) at (2.2*3-2.2*3/2,-3*1.5) {$0$};

\draw[latex-latex, opacity=.3, dashed] (p_0_0) -- (p_1_0);
\draw[latex-latex] (p_0_0) -- (p_1_1);
\draw[latex-latex, opacity=.3, dashed] (p_1_0) -- (p_1_1);
\draw[latex-latex, opacity=.3, dashed] (p_1_0) -- (p_2_0);
\draw[latex-latex, opacity=.3, dashed] (p_1_0) -- (p_2_1);
\draw[latex-latex] (p_1_1) -- (p_2_1);
\draw[latex-latex] (p_1_1) -- (p_2_2);
\draw[latex-latex] (p_2_0) -- (p_2_1);
\draw[latex-latex] (p_2_1) -- (p_2_2);
\draw[latex-latex, opacity=.3, dashed] (p_3_0) -- (p_2_0);
\draw[latex-latex, opacity=.3, dashed] (p_3_0) -- (p_3_1);
\draw[latex-latex] (p_3_1) -- (p_2_0);
\draw[latex-latex] (p_3_1) -- (p_2_1);
\draw[latex-latex, opacity=.3, dashed] (p_3_1) -- (p_3_2);
\draw[latex-latex, opacity=.3, dashed] (p_3_2) -- (p_2_1);
\draw[latex-latex, opacity=.3, dashed] (p_3_2) -- (p_2_2);
\draw[latex-latex, opacity=.3, dashed] (p_3_2) -- (p_3_3);
\draw[latex-latex, opacity=.3, dashed] (p_3_3) -- (p_2_2);

\end{tikzpicture}
\end{figure}

\noindent The neighbourhoods of each state are
\begin{alignat*}{3}
&\mathcal{N}(0) = \{1\}, &&\quad\mathcal{N}(1) = \{0,3,4\}, &&\quad\mathcal{N}(2) = \{3,5\}, \\
&\mathcal{N}(3) = \{1,2,4,5\}, &&\quad\mathcal{N}(4) = \{1,3\}, &&\quad\mathcal{N}(5) = \{2,3\}.
\end{alignat*}
and given how we defined $\mathcal{N}_{b}(i)$ we have
\begin{alignat*}{3}
&\mathcal{N}_0(0) = \emptyset, &&\quad \mathcal{N}_1(0) = \emptyset, &&\quad \mathcal{N}_2(0) = \{1\}, \\
&\mathcal{N}_0(1) = \{0\}, &&\quad \mathcal{N}_1(1) = \{3\}, &&\quad \mathcal{N}_2(1) = \{4\}, \\
&\mathcal{N}_0(2) = \emptyset, &&\quad \mathcal{N}_1(1) = \emptyset, &&\quad \mathcal{N}_2(2) = \{3,5\}, \\
&\mathcal{N}_0(3) = \{1\}, &&\quad \mathcal{N}_1(3) = \{2,5\}, &&\quad \mathcal{N}_2(3) = \{4\}, \\
&\mathcal{N}_0(4) = \{1\}, &&\quad \mathcal{N}_1(4) = \{3\}, &&\quad \mathcal{N}_2(4) = \emptyset, \\
&\mathcal{N}_0(5) = \{2,3\}, &&\quad \mathcal{N}_1(5) = \emptyset, &&\quad \mathcal{N}_2(5) = \emptyset.
\end{alignat*}
The transition matrix obtained through Algorithm \ref{alg:construction_markov_chain} is then equal to
%
%\begin{equation*}
%P = \begin{pmatrix}
%\cdot & \frac{1}{1 \vee \sqrt{2}}p_3 & 0 & 0 & 0 & 0 \\
%\frac{\sqrt{2}}{\sqrt{2} \vee 1}p_1 & \cdot & 0 & \frac{2}{2 \vee 1}p_2 & \frac{1/2}{ 1/2 \vee 1 }p_3 & 0 \\
%0 & 0 & \cdot & \frac{2}{(2+3/4) \vee (1/4+3/4)}p_3 & 0 & \frac{3/4}{(2+3/4) \vee (1/4 + 2)}p_3 \\
%0 & \frac{1}{1 \vee 2}p_1 & \frac{1/4}{(1/4+3/4) \vee (2+3/4)}p_2 & \cdot & \frac{1/2}{1/2 \vee 2}p_3 & \frac{3/4}{(1/4+3/4) \vee (2+1/4)}p_2 \\
%0 & \frac{1}{1 \vee 1/2}p_1 & 0 & \frac{2}{2 \vee 1/2}p_2 & \cdot & 0 \\
%0 & 0 & \frac{1/4}{(1/4+2) \vee (3/4+2)}p_1 & \frac{2}{(1/4+2) \vee (1/4+3/4)}p_1 & 0 & \cdot
%\end{pmatrix}
%\end{equation*}

\begin{equation*}
P = \begin{pmatrix}
\cdot & \frac{1}{\sqrt{2}}p_2 & 0 & 0 & 0 & 0 \\
p_0 & \cdot & 0 & p_1 & \frac{1}{2}p_2 & 0 \\
0 & 0 & \cdot & \frac{8}{11}p_2 & 0 & \frac{3}{11}p_2 \\
0 & \frac{1}{2}p_0 & \frac{1}{11}p_1 & \cdot & \frac{15}{44}p_2 & \frac{1}{3}p_1 \\
0 & p_0 & 0 & p_1 & \cdot & 0 \\
0 & 0 & \frac{1}{11}p_0 & \frac{10}{11}p_0 & 0 & \cdot
\end{pmatrix}
\end{equation*}
where $\cdot$ represents the required quantity so that the rows sum up to 1.
\end{ex}

\subsection{Perfect sampling}
We now introduce an update function for the Markov chain defined in
Proposition \ref{prop:multivariate_markov_chain} so that CFTP is
implementable. Given the current state, a roll of the die $B$, and a draw from a uniform random variable $U$, the update function details where the chain moves next (its formal definition can be found in Appendix \ref{sec:appendix_cftp}). This motivates the choice of defining the transition matrix as the element-wise product of two matrices (cf. equation \eqref{eq:transition_matrix_product}): if the chain is in state $i$, we attempt to move to any state $j$ such that $W_{i,j} = p_B$ and reach a final decision by comparing $U$ and the values of $V_{i,j}$.  We are then able to draw samples from a multivariate
fine and connected ladder and thus solve the original problem via
Theorem~\ref{thm:decomposition_multivariate}. For the general case of
a die with more than 2 faces, the update function defined in the following
proposition is not necessarily monotonic. However, in the Bernoulli Factory
setting of $m=1$, we can define a monotonic update function for the
Markov chain as shown in Corollary
\ref{prop:markov_chain_update_function_univariate}. Notice that even when a monotonic construction is not available, CFTP can still be used in practice. As numerical examples demonstrate (cf. Examples \ref{ex:augmentin_states_faster}, \ref{ex:bernoulli_race}), if the degree of the polynomials and the numbers of faces of the given die are not too large, running times are not prohibitive.

\begin{prop}
\label{prop:multivariate_update_function}
Given a fine and connected ladder $\pi: \Delta^m \rightarrow
\Delta^k$, consider the Markov chain $(X_t)_{t \in \mathbb{N}}$ with transition matrix $P$  of the form~\eqref{eq:transition_matrix_product} and 
defined in Proposition \ref{prop:multivariate_markov_chain}. Let $B \sim \boldsymbol{p}$ and $U
\sim \text{Unif}(0,1)$ be independent
random variables. Given $i \in \Omega$ denote the elements of $\mathcal{N}_B(i)$ as $\mathcal{N}_B(i) = \{j_0,\ldots,j_{w}\}$. Define the function $\phi: \{0,\ldots,k\} \times \{0,\ldots,m\} \times [0,1] \rightarrow \{0,\ldots,k\}$ as
\begin{equation}
\label{eq:update_function_ladder_multivariate}
\phi(i,B,U) = \begin{cases}
j_0 &\quad \text{if } U \leq V_{i,j_0}, \\
j_1 &\quad \text{if } V_{i,j_0} < U \leq V_{i,j_0}+V_{i,j_1}, \\
	&\ldots \\
j_l &\quad \text{if } \sum_{h=0}^{l-1} V_{i,j_h} < U \leq \sum_{h=0}^{l} V_{i,j_h}, \\
	&\ldots \\
i	&\quad \text{otherwise.}
\end{cases}
\end{equation}
Then $\phi$ is an update function for the Markov chain $(X_t)_{t \in \mathbb{N}}$.
\end{prop}

\subsection{A special case: from coins to dice}
%\textcolor{teal}{GM: This section and its proofs are all new or updated.}

Assume that we are given a $p$-coin, and write $p_0
= 1-p$ and $p_1 = p$, to  consider a fine and
connected ladder of the form $\pi: (0,1) \rightarrow \Delta^k$ as in
equation \eqref{eq:univariate_connected_fine_ladder}. %In a Bernoulli
                                %Factory scenario, we would also have
                                %$k=2$. 
In this case the definition of neighbourhoods simplifies and so does
the Markov chain defined in Proposition
\ref{prop:multivariate_markov_chain}. Moreover, given the simplified
structure of the state space, the update function defined in
Proposition \ref{prop:multivariate_update_function} is monotonic, so
that monotonic CFTP can be employed. These observations are summarised in the following Corollary. Figure \ref{fig:univariate_ladder_MC} gives a graphical representation of the dynamics of the Markov chain.

\begin{cor}
\label{prop:markov_chain_update_function_univariate}
Let $\pi: (0,1) \rightarrow \Delta^{k}$ be a fine and connected ladder as in equation \eqref{eq:univariate_connected_fine_ladder}. The transition matrix of the Markov chain $(X_t)_{t \in \mathbb{N}}$ defined in Proposition \ref{prop:multivariate_markov_chain} can be rewritten as 
\begin{equation}
\label{eq:transition_probs_ladder_univariate}
P_{i,j} = \begin{cases}
R_{i-1} \frac{1-p}{R_{i-1} \vee R_i} &\text{if } j = i-1, j > 0, \\
1 - R_{i-1} \frac{1-p}{R_{i-1} \vee R_i} - R_{i+1} \frac{p}{R_i \vee R_{i+1}} &\text{if } j = i, \\
R_{i+1} \frac{p}{R_i \vee R_{i+1}} &\text{if } j = i+1, j < k, \\
0 &\text{otherwise.}
\end{cases}
\end{equation}
Let  $U \sim \text{Unif}(0,1)$ and $p-$coin $B$ be an independent r.v.
(i.e. $\mathbb{P}(B=1)=1-\mathbb{P}(B=0) =p$). The update function defined in Proposition \ref{prop:multivariate_update_function} can be rewritten as  
\begin{equation}
\label{eq:update_function_ladder_univariate}
\phi(i,B,U) = \begin{cases}
i-1 &\qquad \text{if } i>0, B=0, U \leq \frac{R_{i-1}}{R_{i-1} \vee R_i}, \\
i+1 &\qquad \text{if } i<k, B=1, U \leq \frac{R_{i+1}}{R_{i} \vee R_{i+1}}, \\
i &\qquad \text{otherwise.}
\end{cases}
\end{equation}
Moreover, $\phi$ is a monotonic update function for the Markov chain $(X_t)_{t \in \mathbb{N}}$ where $0$ and $k$ are the minimum and maximum states respectively.
\end{cor}

\begin{figure}[h]
\centering
\begin{tikzpicture}
\node (p_0) at (2.5*0-2.5*3/2,0) {$R_0\frac{(1-p)^4}{C(p)}$};
\node (p_1) at (2.5*1-2.5*3/2,0) {$R_1\frac{p(1-p)^3}{C(p)}$};
\node (p_2) at (2.5*2-2.5*3/2,0) {$R_2\frac{p^2(1-p)^2}{C(p)}$};
\node (p_3) at (2.5*3-2.5*3/2,0) {$R_3\frac{p^3(1-p)}{C(p)}$};
\node (p_4) at (2.5*4-2.5*3/2,0) {$R_4\frac{p^4}{C(p)}$};

\draw[-latex] (p_0) to [loop above] node[midway, above]{$P_{0.0}$}(p_0);
\draw[-latex] (p_0) to [bend left] node[midway, above]{$P_{0,1}$}(p_1);
\draw[-latex] (p_1) to [bend left] node[midway, below]{$P_{1,0}$}(p_0);

\draw[-latex] (p_1) to [loop above] node[midway, above]{$P_{1,1}$}(p_1);
\draw[-latex] (p_1) to [bend left] node[midway, above]{$P_{1,2}$}(p_2);
\draw[-latex] (p_2) to [bend left] node[midway, below]{$P_{2,1}$}(p_1);

\draw[-latex] (p_2) to [loop above] node[midway, above]{$P_{2,2}$}(p_2);
\draw[-latex] (p_2) to [bend left] node[midway, above]{$P_{2,3}$}(p_3);
\draw[-latex] (p_3) to [bend left] node[midway, below]{$P_{3,2}$}(p_2);

\draw[-latex] (p_3) to [loop above] node[midway, above]{$P_{3,3}$}(p_3);
\draw[-latex] (p_3) to [bend left] node[midway, above]{$P_{3,4}$}(p_4);
\draw[-latex] (p_4) to [bend left] node[midway, below]{$P_{4,3}$}(p_3);

\draw[-latex] (p_4) to [loop above] node[midway, above]{$P_{4,4}$}(p_4);
\end{tikzpicture}
\caption{Transition probabilities on a fine and connected univariate ladder.}
\label{fig:univariate_ladder_MC}
\end{figure}
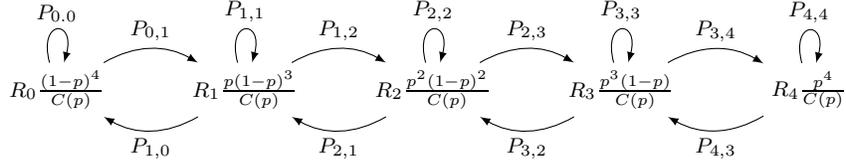

\subsection{Efficiency of the algorithm}
\label{sec:efficiency}
We now provide some results on the expected number of rolls of the original die required by CFTP and give insights on how the algorithm can be sped up. In particular, we  provide conditions for the expected number of rolls to be bounded uniformly in  $\boldsymbol{p}$. Interestingly, this is always the case when $m=1$, thus also in the Bernoulli Factory scenario. Moreover, we give tighter bounds when the univariate ladder is strictly log-concave, as defined below, and show that univariate ladders can be always transformed into an equivalent log-concave ladder through augmentation.

\begin{defi}[Log-concave discrete distribution]
\label{def:logconcave_discrete_distribution}
A discrete distribution $\boldsymbol{\mu}$ on $\Omega = \{0,\ldots,k\}$ is log-concave if for all $0 < i < k$,
\begin{equation} \label{eq:lc}
\mu_i^2 \geq \mu_{i-1}\mu_{i+1}.
\end{equation}
If the inequality is strict, then $\boldsymbol{\mu}$ is said to be strictly log-concave.
\end{defi}
 
We shall also show that augmenting the degree of the ladder produces a
log-concave distribution out of a ladder that is not itself
log-concave. To that end, we provide the following general theorem
that is of independent interest.
 
\begin{thm}
\label{conj:logconcavity_convolution_binomial}
For every discrete random variable $W$ on $\Omega = \{0,\ldots,n_0\}$ where $n_0 < \infty$, there exists a number $n = n(W)$ such that $W + B_n$ is strictly log-concave, where $B_n$ is an independent Binomial$(n,1/2)$.
\end{thm}

Proof of the theorem is deferred to the appendix. Several works have looked into whether log-concavity is preserved
\cite{Hoggar1974, Saumard2014}, but checking whether some operations
\textit{introduce} log-concavity seems to be a harder problem
\cite{Johnson2006} and the above result appears to be the first in
this direction.
 % If $\pi$ is strictly log-concave, the result follows trivially, as $B_n$ is log-concave and log-concavity is preserved under convolution \cite{Hoggar1974}. However, since $B_n$ is log-concave and $\pi+B_n$ approaches a Gaussian in the limit, which is itself log-concave. Indeed, as $n$ grows, the influence of $\pi$ on the distribution of $\pi+B_n$ becomes negligible, since $\pi$ has finite support. 

Building on Theorem \ref{conj:logconcavity_convolution_binomial}, we will
show that augmenting the degree of the ladder may lead to faster
implementation. This is expanded in
Proposition~\ref{lemma:impose_efficiency_condition} and empirically
verified in Examples \ref{ex:increase_degree_ladder} and \ref{ex:augmentin_states_faster}.

\subsubsection{General case}

\begin{prop}
\label{prop:bound_expected_rolls}
Let $\pi(\boldsymbol{p}): \Delta^m \rightarrow \Delta^k$ be a fine and connected ladder of degree $d$. Write $\boldsymbol{n}_i$ as in Definition \ref{def:multivariate_ladder} and assume $E := \{b \in \{0,\ldots,m\}: \exists i,  n_{i,b} = d \}$ is nonempty. Denote by $N$ the number of rolls of the original die required by CFTP when the update function of Proposition \ref{prop:multivariate_update_function} is used. Then, one can explicitly construct a new ladder $\pi'(\boldsymbol{p}): \Delta^m \rightarrow \Delta^w$ of degree $2d$ where $w < \min\{k(m+1)^d, {{2d+m}\choose{m}}\}$ that is a disaggregation of $\pi$ and such that  
\begin{equation*}
\mathbb{E}[N] \leq \min_{b\in E} \frac{(ap_b)^{-2d}-1}{1-ap_b},
\end{equation*}
where $a \in (0,1]$ is a constant independent on $\boldsymbol{p}$.
\end{prop}

If $E = \{0,\ldots,m\}$, then we can bound $N$ by a quantity independent of $\boldsymbol{p}$, that is
\begin{equation}
\label{eq:loose_bound_m_2}
\mathbb{E}[N] \leq \min_{b\in E} \frac{(ap_b)^{-2d}-1}{1-ap_b} \leq \frac{ \left(\frac{a}{m+1}\right)^{-2 d} - 1}{1 - \frac{a}{m+1}}
\end{equation}

\subsubsection{From coins to dice}
We now restrict our analysis to rational
functions of the form $f: (0,1) \rightarrow \Delta^v$, where an implementation of monotonic CFTP is possible. We  study the
efficiency of the proposed method in terms of the required number of
tosses of the given $p$-coin. A direct consequence of Proposition \ref{prop:bound_expected_rolls} is the following.
\begin{cor}
\label{cor:bound_expected_tosses}
Let $\pi(p): (0,1) \rightarrow \Delta^k$ be a fine and connected ladder. Denote by $N$ the number of tosses of the $p$-coin required by CFTP when the update function of Corollary \ref{prop:markov_chain_update_function_univariate} is used. Then, one can explicitly construct a new ladder $\pi'(\boldsymbol{p}): (0,1) \rightarrow \Delta^{2k}$ that is a disaggregation of $\pi$ and such that 
\begin{equation*}
\mathbb{E}[N] \leq \min\left\{ \frac{(ap)^{-2(k-1)}-1}{1-ap}, \frac{(a(1-p))^{-2(k-1)}-1}{1-a(1-p)} \right\} \leq \frac{ \left(\frac{a}{2}\right)^{-2(k-1)} - 1}{1 - \frac{a}{2}}
\end{equation*}
where $a = \min_i\left\{\frac{R_{i-1}}{R_{i-1} \vee R_i} \wedge \frac{R_i}{R_{i-1} \vee R_i}\right\} \in (0,1]$ is a constant independent of $\boldsymbol{p}$.
\end{cor}

Therefore, the expected running time of the algorithm can be bounded uniformly in $p$. However, the bound proposed in Corollary \ref{cor:bound_expected_tosses} is
generally very loose and does not give insights into how the algorithm
could potentially be sped up. We now provide a tighter bound under the condition that the ladder $\pi(p)$ is a log-concave distribution. 

The proof of the following Proposition is in spirit similar to the Path Coupling technique of \cite{Bubley1997}.

\begin{prop}
\label{prop:efficiency_monotonic_CFTP}
Let $\pi:(0,1) \rightarrow \Delta^k$ be a univariate fine and
connected ladder as in
\eqref{eq:univariate_connected_fine_ladder}. Assume further that $\pi$ is strictly log-concave and that the Markov chain and update function defined in Corollary
\ref{prop:markov_chain_update_function_univariate} are
used. Then 
\begin{eqnarray*}
\mathbb{P}(N \geq n) & \leq & (k-1)\rho^n \qquad \text{and} \\
\mathbb{E}[N] & \leq & \frac{(k-1)\rho}{1-\rho},
\end{eqnarray*} 
where $\rho \in (0,1)$ for all $p \in (0,1)$ is given by
\begin{equation}
\label{eq:rho_efficiency}
\rho = \max_{i\in\{0,\ldots,k-2\}} [1-(P_{i,i+1}-P_{i+1,i+2}) - (P_{i+1,i}-P_{i,i-1})]
\end{equation}
with $P_{i,j}$ given by
\eqref{eq:transition_probs_ladder_univariate} with the convention that
$P_{k,k+1} = P_{0,-1} = 0$.
\end{prop}

\begin{remark}
Given a univariate fine and connected ladder, $\rho$ as in equation \eqref{eq:rho_efficiency} can be explicitly computed for a fixed $p \in (0,1)$.  Moreover, the algorithm still requires a finite number of tosses even if $p=1$ or $p=0$. In this case the expected number of tosses $\mathbb{E}[N]$ can be exactly computed:
\begin{equation*}
\mathbb{E}[N] = \begin{cases}
\frac{R_0 \vee R_1}{R_1} + \ldots + \frac{R_{k-2} \vee R_{k-1}}{R_{k-1}} &\quad \text{if } p = 1, \\
\frac{R_0 \vee R_1}{R_0} + \ldots + \frac{R_{k-2} \vee R_{k-1}}{R_{k-2}} &\quad \text{if } p = 0. \\
\end{cases}
\end{equation*}
\end{remark}

Given a rational function $f: (0,1) \rightarrow \Delta^v$, we have
proved in Theorem \ref{thm:decomposition_multivariate} that it is
always possible to construct a univariate fine and connected ladder
$\pi: (0,1) \rightarrow \Delta^k$. However, $\pi$ may not be strictly
log-concave. Using Theorem \ref{conj:logconcavity_convolution_binomial}, we now show that augmenting the ladder enough times produces a new ladder that is strictly log-concave, so that Proposition \ref{prop:efficiency_monotonic_CFTP} applies.

\begin{lemma}
\label{lemma:impose_efficiency_condition}
Let $\pi:(0,1) \rightarrow \Delta^k$ be a univariate fine and
connected ladder as in Section \ref{def:univariate_connected_fine_ladder}. Then, one can explicitly construct a new univariate fine and connected ladder $\pi':(0,1) \rightarrow \Delta^w$ where $w \geq k$ such that $\pi'$ is a disaggregation of $\pi$ and $\pi'$ is strictly log-concave.
\end{lemma}

Hence increasing the degree of a generic
univariate fine and connected ladder $\pi: (0,1) \rightarrow \Delta^k$
may lead to a faster implementation of monotonic CFTP despite an increased number of states that the chain
needs to visit. Clearly, this leads to a trade-off that the user may want
to calibrate, as shown in Examples \ref{ex:increase_degree_ladder} and \ref{ex:augmentin_states_faster}.

\begin{ex}
\label{ex:increase_degree_ladder}
Consider sampling from the following ladder
\begin{equation*}
\pi(p) \propto ((1-p)^4,1000p(1-p)^3,p^2(1-p)^2,500p^3(1-p),p^4).
\end{equation*}
Clearly, $\pi$ is not log-concave. We can augment the ladder up to two times to obtain respectively
\begin{align*}
\pi^{(1)} &\propto ((1-p)^5,1001p(1-p)^4,1001p^2(1-p)^3,501p^3(1-p)^2,500p^4(1-p),p^5)  \\
\pi^{(2)} &\propto ((1-p)^6,1002p(1-p)^5,2002p^2(1-p)^4, \\ 
&1502p^3(1-p)^3,1001p^4(1-p)^2,500p^5(1-p),p^6).
\end{align*}
Notice that now $\pi^{(2)}$ is strictly log-concave. Table \ref{tbl:efficiency_tosses} shows the empirical number of tosses required by the algorithm when sampling from either $\pi(p)$, $\pi^{(1)}(p)$ or $\pi^{(2)}(p)$ for different values of $p$. Notice that even if $\pi^{(1)}(p)$ is not log-concave, it still leads to a slightly faster implementation than when targeting $\pi^{(2)}(p)$. 

\begin{table}[H]
\begin{tabular}{c|ccccccc}
    & \multicolumn{7}{c}{True value of $p$}                    \\
    & 0.01  & 0.1   & 0.25   & 0.5   & 0.75  & 0.9   & 0.99  \\ \hline
%$\pi$  & 556 (1157)   & 631.1 (1243) & 861.2 (1694)  & 1339 (2679)  & 1446 (2927)  & 1272 (2443)  & 1138 (2293) \\
%$\pi_1$ & 90.86 (97.47) & 12.16 (17.18) & 7.672 (13.61)  & 8.555 (16.36) & 9.298 (17.45) & 12.75 (20.16) & 82.73 (92.52) \\
%$\pi_2$ & 91.74 (103.7) & 13.25 (23.31) & 9.355 (19.93) & 11.16 (23.04) & 11.68 (24.01) & 14.78 (26.41) & 82.51 (98.84)
\\[-1em]
$\hat{\mathbb{E}}[N_\pi]$  & 561.31    & 621.73 & 827.86  & 1332.63   & 1433.59  & 1209.28   & 1090.54   \\ \\[-1em] \hline
\\[-1em]
$\hat{\mathbb{E}}[N_{\pi^{(1)}}]$ & 92.35 & 12.21 & 7.72 & 8.65 & 9.48 & 12.59 & 87.05 \\ \\[-1em] \hline
\\[-1em]
$\hat{\mathbb{E}}[N_{\pi^{(2)}}]$ & 93.17 & 13.33 & 9.43 & 11.29 & 11.67 & 14.47 & 89.56 \\ \\[-1em] \hline
\end{tabular}
\caption{Average number of required tosses of the $p$-coin over 1,000 runs of the algorithm when targeting $\pi, \pi^{(1)}$ and $\pi^{(2)}$.}
\label{tbl:efficiency_tosses}
\end{table}

\end{ex}

\section{Examples and implementation}
\label{sec:examples_implementation}

An \texttt{R} package implementing the method and reproducing the examples is available at \url{https://github.com/giuliomorina/DiceEnterprise}. The user is just required to define the function $f(\boldsymbol{p})$ and to provide a function that rolls the original die. Then, the package automatically constructs the fine and connected ladder and implements CFTP. If the original die has only two faces, the monotonic version of CFTP is automatically employed. 
%As rolling the given die may be computationally expensive, the package also supports a version of CFTP where instead of doubling the time step at each iteration, it is just increased by one.

We now show how the method works and performs on some examples, all of which can be reproduced using the provided package. We
start with a toy example to better explain and highlight
the construction proposed in Theorem
\ref{thm:decomposition_multivariate} and Proposition
\ref{prop:multivariate_markov_chain}. Next, we examine efficiency of
the monotonic and general versions of the algorithm by considering
higher order rational functions. We also consider the so-called
logistic Bernoulli factory as studied in \cite{Huber2017}. We show
that our method leads to a simple algorithm which on average requires
the same number of tosses as the approach of
\cite{Huber2017}. Finally, we deal with a slightly different scenario
where instead of an $m$-sided die, $m$ independent coins are provided
where the probability $\boldsymbol{p} = (p_0,\ldots,p_{m-1}) \in (0,1)^m$
of tossing heads is unknown. In particular, we notice how we can
construct a Dice Enterprise for the \q{Bernoulli Race} function
considered in \cite{Dughmi2017} which again has the same performance in terms of the expected number of required tosses.

\begin{ex}[Toy example of Bernoulli Factory]
\label{ex:toy_bf}
Let $p \in (0,1)$ and assume we wish to generate a coin that lands heads with probability 
\begin{equation*}
\frac{\sqrt{2}p^3}{(\sqrt{2}-5)p^3+11p^2-9p+3},
\end{equation*}
having access only to a $p$-coin. Our proposed construction produces the following fine and connected ladder
\begin{equation*}
\pi(p) = (3(1-p)^4,3p(1-p)^3,2p^2(1-p)^2,(\sqrt{2}+2)p^3(1-p),\sqrt{2}p^4),
\end{equation*}
via the following steps:
\begin{enumerate}
\item Let $C(p) = (\sqrt{2}-5)p^3+11p^2-9p+3$ and consider 
\begin{equation*}
f(p) =  \frac{1}{C(p)}(\underbrace{-5p^3+11p^2-9p+3}_{D_0(p)}, \underbrace{\sqrt{2}p^3}_{D_1(p)}).
\end{equation*}
Convert $D_0(p)$ and $D_1(p)$ into homogeneous polynomials in the variables $p$ and $(1-p)$ with positive coefficients and of the same degree. This can be achieved by using the multinomial theorem (cf. proof of Theorem \ref{thm:decomposition_multivariate}). We get
\begin{align*}
D_0(p) &= 2p^2(1-p)+3(1-p)^3, \\
D_1(p) &= \sqrt{2}p^3, \\
\end{align*}
so that we can equivalently consider the ladder
\begin{equation*}
\pi'(p) = \frac{1}{C(p)}(3(1-p)^3,2p^2(1-p),\sqrt{2}p^3).
\end{equation*}
and notice that if $X \sim \pi'$, then $W = \mathbb{I}(X \in \{3\})$ is distributed as $f(p)$.
\item Notice that $\pi'$ is not a connected ladder, as there is no term proportional to $p(1-p)^2$. By applying the binomial theorem, we can construct a new ladder
\begin{equation*}
\tilde{\pi}(p) = \frac{1}{C(p)} (\tilde{\pi}_0(p),\tilde{\pi}_1(p),\tilde{\pi}_2(p),\tilde{\pi}_3(p),\tilde{\pi}_4(p),\tilde{\pi}_5(p)),
\end{equation*}
where 
\begin{align*}
\tilde{\pi}_0(p) &= \pi'_0(p){{1}\choose{0}}(1-p) = 3(1-p)^4, \\
\tilde{\pi}_1(p) &= \pi'_0(p){{1}\choose{1}}p = 3p(1-p)^3, \\
\tilde{\pi}_2(p) &= \pi'_1(p){{1}\choose{0}}(1-p) = 2p^2(1-p)^2, \\
\tilde{\pi}_3(p) &= \pi'_1(p){{1}\choose{1}}p = 2p^3(1-p), \\
\tilde{\pi}_4(p) &= \pi'_2(p){{1}\choose{0}}(1-p) =  \sqrt{2}p^3(1-p), \\
\tilde{\pi}_5(p) &= \pi'_2(p){{1}\choose{1}}p = \sqrt{2}p^4. \\
\end{align*}
Notice that if $Y \sim \tilde{\pi}$, then $X = \mathbb{I}(Y \in \{2,3\}) + 2\cdot \mathbb{I}(Y \in \{4,5\})$ is distributed as $\pi'$.
\item Finally, we can construct a fine and connected ladder by adding up together the terms where the same monomial appears:
\begin{equation*}
\pi(p) = \frac{1}{C(p)}(3(1-p)^4,3p(1-p)^3,2p^2(1-p)^2,(\sqrt{2}+2)p^3(1-p),\sqrt{2}p^4).
\end{equation*}
Assume $Z \sim \pi$, $U \sim \text{Unif}(0,1)$ and let
\begin{align*}
Y &= \cdot\mathbb{I}(Z=1) + 2\cdot \mathbb{I}(Z=2) +\\
  &3\cdot\mathbb{I}\left( Z=3, U\leq\frac{2}{2+\sqrt{2}} \right) 
+ 4\cdot \mathbb{I}\left( Z=3, U > \frac{2}{2+\sqrt{2}} \right) + 5\cdot \mathbb{I}(Z=4)
\end{align*}
so that $Y\sim \tilde{\pi}$.
\end{enumerate}

Table \ref{tbl: bernoulli_factory} shows the performance of CFTP for different values of the unknown probability $p$.

\begin{table}[H]
\centering
% \begin{tabular}{l|ccccccc}
%   \hline
%  $p$ & 0.01 & 0.1 & 0.25 & 0.5 & 0.75 & 0.9 & 0.99 \\ 
%   \hline
% True value $f(p)$ & 0.00 & 0.00 & 0.02 & 0.22 & 0.65 & 0.86 & 0.99 \\ 
%   Lower C.I. & 0.00 & 0.00 & 0.01 & 0.18 & 0.63 & 0.86 & 0.98 \\ 
%   Estimated $\hat{f}(p)$ & 0.00 & 0.00 & 0.01 & 0.21  & 0.66 & 0.88 & 0.99  \\ 
%   Upper C.I. & 0.00 & 0.00 & 0.02 & 0.23 & 0.69 & 0.90 & 0.99 \\ 
%   $\hat{\mathbb{E}}[N]$ & 4.80 & 5.61 & 7.45 & 10.61 & 8.05 & 6.51 & 5.94 \\ 
%   \hline
% \end{tabular}
\begin{tabular}{l|ccccccc}
  \hline
 $p$ & 0.01 & 0.25 & 0.5 & 0.75 & 0.99 \\ 
  \hline
  $f(p)$ & 0.00  & 0.02 & 0.22 & 0.65  & 0.99 \\ 
  $\hat{f}(p)$ & $\cithree{0.00}{0.00}{0.00}$  & $\cithree{0.01}{0.01}{0.02}$ & $\cithree{0.18}{0.21}{0.03}$  & $\cithree{0.63}{0.66}{0.69}$ & $\cithree{0.98}{0.99}{0.99}$ \\ 
  $\hat{\mathbb{E}}[N]$ & 4.80  & 7.45 & 10.61 & 8.05 & 5.94 \\ 
   \hline
\end{tabular}
\caption{Implementation of the Bernoulli Factory for the function $f(p) = \frac{\sqrt{2}p^3}{(\sqrt{2}-5)p^3+11p^2-9p+3}
$ and for different values of the true unknown probability $p$. The algorithm has been run 1,000 times to obtain tosses of the $f(p)$-coin and $\hat{f}(p)$ is the sample average. Smaller numbers represent 95\% confidence interval computed via the method of \cite{Sison1995}. $\hat{\mathbb{E}}[N]$ is the empirical expected number of tosses in a run of the CFTP algorithm. }
\label{tbl: bernoulli_factory}
\end{table}
\end{ex}

\begin{ex}[Augmenting the number of states can lead to faster running time]
\label{ex:augmentin_states_faster}
Given a $p$-coin, consider constructing a 3-sided die where the probability of rolling each face is given by
\begin{equation}
\label{eq:example_efficiency_coin}
\pi(p) \propto \{p^{20},p^{10}(1-p)^{10},(1-p)^{20}\}.
\end{equation}
A naive rejection sampling approach to construct a Bernoulli Factory for $\pi(p)$ would be the following: toss the $p$-coin 20 times and with probability $1/3$ output $1$ if all the tosses are heads, with probability $1/3$ output $2$ if the first 10 tosses are heads and the last 10 tosses are tails, with probability $1/3$ output $3$ if all the tosses are tails. In all other cases, restart the algorithm. 

Assume now that $p = 1/2$, so that the expected number of tosses of this naive procedure would be $\mathbb{E}[N] = 2^{20} \approx 10^6$. Table \ref{tbl:example_efficiency_coin} shows the performance of our novel algorithm on the same example when targeting the ladder of equation \eqref{eq:example_efficiency_coin}, as well as when targeting the augmented ladder where extra states are added. Indeed, Lemma \ref{lemma:impose_efficiency_condition} and Proposition \ref{prop:efficiency_monotonic_CFTP} suggest that doing so may lead to faster performance, as empirically confirmed. Notice that to get a strictly log-concave ladder, we need to augment $\pi$ at least 203 times. In practice, it is enough to augment it around 40 times to obtain optimal performance, due to the trade-off effect discussed in Section \ref{sec:efficiency}.

\begin{table}[H]
\begin{tabular}{c|ccccccc}
    & \multicolumn{7}{c}{Number of states added to the original ladder $\pi(p)$}                    \\
    & +0  & +20   & +40   & +60   & +80  & +100   & +120  \\ \hline
    \\[-1em]
$\hat{\mathbb{E}}[N]$  & 5337.7    & 585.7 & 471.4  & 481.7  & 529.3  & 590.4   & 647.9   \\ \hhline{========}
& +140  & +160   & +180   & +200   & +220  & +240   & +260 \\ \hline
\\[-1em]
$\hat{\mathbb{E}}[N]$  & 717.2    & 774.2 & 840.2  & 892.3   & 927.3  & 996.4   & 1038.9  \\ 
\end{tabular}
\caption{Implementation of the Dice Enterprise for the function of eq. \eqref{eq:example_efficiency_coin} when $p = 1/2$. The algorithm has been run 1,000 times and $\hat{\mathbb{E}}[N]$ is the empirical number of tosses of the $p$-coin required. The augmented ladder is strictly log-concave when at least 203 states are added.}
\label{tbl:example_efficiency_coin}
\end{table}

Consider now a slightly different example, where a 3-sided fair die is given, i.e. $\boldsymbol{p} = (1/3,1/3,1/3)$, and the aim is to construct a 4-sided die where the probability of rolling each face is given by 
\begin{equation}
\label{eq:example_efficiency_die}
\pi(\boldsymbol{p}) \propto \{p_0^5p_1^5p_2^5,p_0^{15},p_1^{15},p_2^{15}\}.
\end{equation}
A naive approach as the one before would require on average $\mathbb{E}[N] = 3^{15}\approx 1.4 \times 10^7$ rolls of the $\boldsymbol{p}$-die. Although the result of Proposition \ref{prop:efficiency_monotonic_CFTP} does not hold here, as a monotonic implementation of CFTP is not possible, augmenting the ladder may still lead to faster performance. This is indeed the case, as shown in Table \ref{tbl:example_efficiency_die}, where targeting the ladder with 60 extra states leads to an implementation that requires on average around 840 tosses of the $\boldsymbol{p}$-die, instead of more than 100,000 when directly targeting the original $\pi(\boldsymbol{p})$.

\begin{table}[H]
\begin{tabular}{c|cccccccc}
    & \multicolumn{8}{c}{Number of states added to the original ladder $\pi(\boldsymbol{p})$}                    \\
    & +0  & +10   & +20   & +30   & +40  & +50   & +60 & +70  \\ \hline
    \\[-1em]
$\hat{\mathbb{E}}[N]$  & 174246.4   & 2569.0 & 1341.5  & 1032.9   & 912.5  & 874.0   & 841.4 & 860.1   \\ \\[-1em] \hline
\end{tabular}
\caption{Implementation of the Dice Enterprise for the function of eq. \eqref{eq:example_efficiency_die} when $\boldsymbol{p} = (1/3,1/3,1/3)$. The algorithm has been run 1,000 times and $\hat{\mathbb{E}}[N]$ is the empirical number of rolls of the $\boldsymbol{p}$-die required.}
\label{tbl:example_efficiency_die}
\end{table}
\end{ex}

\begin{ex}[Logistic Bernoulli Factory]

Consider constructing a Bernoulli factory for the function
\begin{equation*}
\frac{Cp}{1+Cp}, \qquad C > 0.
\end{equation*}
Such problem is considered in \cite{Huber2017} where it is referred as
constructing a logistic Bernoulli factory. In the same paper, the
author proposes an ad-hoc algorithm that exploits properties of
thinned Poisson processes and requires on average $\mathbb{E}[N_H] =
C/(1+Cp)$ tosses of the $p$-coin. We now show that our proposed method
leads to an alternative algorithm that requires on average the same
number of tosses. The fine and connected ladder for this target is
\begin{equation*}
\pi(p) = \frac{1}{1+Cp}((1+C)p,(1-p)).
\end{equation*}
Given $Y \sim \pi$ and $U \sim \text{Unif}(0,1)$, we output heads if
$Y=1, U<C/(1+C)$ and tails otherwise. Sampling from $\pi(p)$ boils
down to sampling from the stationary distribution of a Markov chain
consisting of only two states, as depicted in Figure
\ref{fig:logistic_bf}. CFTP needs to keep track of only two chains
starting in the two states and the algorithm stops as soon as one of
the two chain moves, as they cannot both move at the same time. In particular, the particles coalesce if heads is tossed or if the uniform r.v. $U$ drawn by the algorithm is such that $U \leq 1/(1+C)$. Therefore, CFTP is equivalent to algorithm \ref{alg:logistic_bf} which is a special case of the 2-coin algorithm presented in \cite{Goncalves2017,Goncalves2017b} with $c_1 = C, c_2=1, p_1 = p, p_2 = 1$.

\begin{algorithm}[H]
	\caption{Logistic Bernoulli Factory}
	\label{alg:logistic_bf}
	 \hspace*{\algorithmicindent}\justifying\textbf{Input:} black box to sample from Ber$(p)$, a constant $C > 0$. \\
	 \hspace*{\algorithmicindent}\justifying\textbf{Output:} a sample from Ber$(Cp/(1+Cp))$.
	\begin{algorithmic}[1]
		\State{Sample $U \sim \text{Unif}(0,1)$}
		\If{$U \leq \frac{1}{1+C}$} set $Y:=0$
		\Else
			\State{Sample $B \sim \text{Bern}(p)$}
			\If{$B = 1$} set $Y:=1$
			\Else{ discard $U,B$ and GOTO 1}
			\EndIf
		\EndIf
		\State{Output $Y$}
	\end{algorithmic}
\end{algorithm}

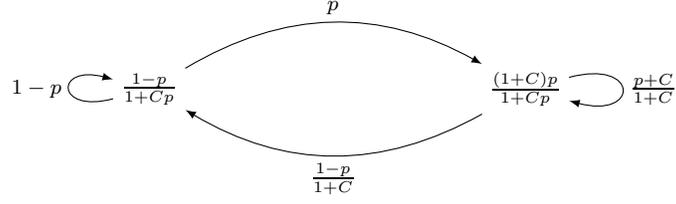
\begin{figure}[H]
\centering
\begin{tikzpicture}
\node (p_0) at (2.5*0-2.5*3/2,0) {$\frac{1-p}{1+Cp}$};
\node (p_1) at (2.5*2-2.5*3/2,0) {$\frac{(1+C)p}{1+Cp}$};

\draw[-latex] (p_0) to [loop left] node[]{$1-p$}(p_0);
\draw[-latex] (p_0) to [bend left] node[midway, above]{$p$}(p_1);
\draw[-latex] (p_1) to [bend left] node[midway, below]{$\frac{1-p}{1+C}$}(p_0);
\draw[-latex] (p_1) to [loop right] node[]{$\frac{p+C}{1+C}$}(p_1);

\end{tikzpicture}
\caption{Dynamic of the Markov chain with stationary distribution $\pi(p) = \frac{1}{1+Cp}((1+C)p,(1-p))$.}
\label{fig:logistic_bf}
\end{figure}

The probability that the algorithm stops at a specific iteration is $\frac{1+Cp}{1+C}$. Since each iteration is independent of the others and the probability that a toss of the $p$-coin is required is $\frac{C}{1+C}$, the average number of tosses is then given by
\begin{equation*}
\mathbb{E}[N_{\text{CFTP}}] = \frac{1+C}{1+Cp}\frac{C}{1+C} = \frac{C}{1+Cp}
\end{equation*}
and it is thus equal to $\mathbb{E}[N_H]$.

\end{ex}

\begin{ex}[Independent coins and Bernoulli Race]
\label{ex:bernoulli_race}

We now deal with a slightly different scenario where instead of having access to a die, $m$ independent coins are given. Similarly, the probability of tossing heads on each of the coin is unknown and given by $\boldsymbol{p} = (p_0,\ldots,p_{m-1}) \in (0,1)^m$, so that the problem is now obtaining a sample from a rational function $f: (0,1)^m \rightarrow \Delta^v$. There are several ways to transform tosses of $m$ coins into a roll of a die. In particular, we can construct an $(m+1)$-sided die in the following fashion. Firstly, we choose uniformly which coin to toss, say the $i$\textsuperscript{th}. If the result is heads we output $i \in \{0,\ldots,m-1\}$, otherwise we output $m$. The probabilities of obtaining each face by rolling the so constructed $(m+1)$-sided die is then given by $\tilde{\boldsymbol{p}} = \left(\frac{p_0}{m},\ldots,\frac{p_{m-1}}{m},1-\frac{1}{m}\sum_{i=0}^{m-1} p_i\right) \in \Delta^m$. The function $f(\boldsymbol{p})$ can be transformed into a function of $\tilde{\boldsymbol{p}}$ by substituting $p_i = m\tilde{p_i}$.

We now consider the function $f(\boldsymbol{p}) = \frac{1}{\sum_{i=1}^m p_i}(p_1,\ldots,p_m)$ as in \cite{Dughmi2017}, where the problem of tossing such $f(\boldsymbol{p})$-die is named Bernoulli Race. Their proposed algorithm requires on average $\mathbb{E}[N_D] = m/\sum_{i=1}^m p_i$ tosses of the $m$ coins. After applying the required variable transformation, we then consider $f(\tilde{\boldsymbol{p}}) = \frac{1}{\sum_{i=0}^{m-1} \tilde{p}_i}(\tilde{p}_0,\ldots,\tilde{p}_{m})$ and we can then employ our Dice Enterprise methodology. In this particular problem, $f(\tilde{\boldsymbol{p}})$ is already a multivariate ladder and the transition matrix of the chain constructed as in Proposition \ref{prop:multivariate_markov_chain} is given by

\begin{equation*}
P = \begin{pmatrix}
1-\sum_{i \neq 0} \tilde{p}_i & \tilde{p}_1 & \tilde{p}_2 & \ldots & \tilde{p}_{m+1} \\
\tilde{p}_0 & 1-\sum_{i \neq 1} \tilde{p}_i & \tilde{p}_2 & \ldots & \tilde{p}_{m+1} \\
\vdots & \vdots & \vdots & \vdots & \vdots \\
\tilde{p}_0 & \tilde{p}_1 & \tilde{p}_2 & \ldots & 1-\sum_{i \neq m} \tilde{p}_i \\
\end{pmatrix}
\end{equation*}
Notice that CFTP terminates as soon as either $0,1,\ldots,m-1$ is rolled and continues only when the outcome of the roll is the $m$\textsuperscript{th} face. In this case, each iteration of CFTP is independent of the other and the probability that the algorithm stops is given by
\begin{equation*}
\mathbb{P}(\text{CFTP stops}) = 1 - \frac{1}{m}\sum_{i=1}^m \mathbb{P}(\text{i\textsuperscript{th} coin returns tails}) = \frac{1}{m}\sum_{i=0}^{m-1} p_i 
\end{equation*}
so that $\mathbb{E}[N_{CFTP}] = m/\sum_{i=0}^{m-1} p_i$ and the algorithm is actually equivalent to the one proposed in \cite{Dughmi2017}.

\end{ex}

\section{Conclusions}

The Dice Enterprise algorithm introduced in this paper is a
generalisation of the celebrated Bernoulli Factory algorithms to
rational mappings of categorical distributions. It offers a fully
automated procedure that does not require further user intervention or
case specific design tweaks. Furthermore, in the ``coin to dice'' case
the efficiency of the algorithm can be automatically boosted by
increasing the degree of the target polynomials until the distribution
is log-concave which guarantees fast convergence. The version we developed is based on
Coupling From the Past and enjoys an efficient monotonic
implementation in the ``coin to
dice'' case, however CFTP can be replaced by any other Markov chain
perfect sampling routine, including Fill's interruptible
algorithm. We
demonstrated that several specialised Bernoulli factory algorithms
introduced in literature, such as the two coin algorithm, the logistic
Bernoulli factory or the Bernoulli race can be regarded as special
versions of the Dice Enterprise. 
A natural open problem that follows
from this paper is to design a monotone version of the Dice Enterprise
in the ``dice to dice'' scenario. Further studies may also look into providing bounds for the degree of the decomposition of rational functions into ladders (based on P\'olya positive homogeneous polynomial theorem \cite{Polya1928, Powers2001}) or the number of Bernoulli trials needed to introduce log-concavity when convoluted with a discrete random variable. Understanding of these questions is necessary for obtaining more precise upper bounds on running time of the algorithm. Computing lower bounds for the expected number of rolls required by CFTP -- perhaps through information criteria (c.f. \cite{Mendo2016}) or building on \cite{Huber2016, Karp2000} - would complement the theoretical analysis.

Another question of particular interest is establishing the relation between our approach and that of \cite{Mossel2005}. Both theirs and our work builds on Polya's theorem on homogeneous polynomials, which ensures positivity of all the coefficients and therefore allows for a construction of an equivalent probability distribution in a form that is amenable to simulation via a carefully designed finite automaton, or Markov chain, respectively. The focus of \cite{Mossel2005} is on the theoretical side of characterising distributions through automata rather than on practical algorithm design or efficiency. In particular, the block simulation considered in their Theorem~2.2, and outlined in Proposition~2.5, is closely related to what we termed the naive rejection sampling approach in Example \ref{ex:augmentin_states_faster}. Its cost would scale exponentially in the degree of the involved polynomials. On the other hand, \cite{Mossel2005} poses an open question (Problem 4.1) and asks what is the smallest size of an automaton that simulates $f(p),$ and how to find it. While we do not know the answer to this problem, we conjecture that when $f$ is a rational function with coefficients in $\mathbb{Q},$  the CFTP procedure we designed for simulating a fine and connected ladder $\pi(\boldsymbol{p}): \Delta^m \rightarrow \Delta^k$, is a finite automaton 
with the set of states $S = \{0, \dots, k\} \times \{0, \dots, k\}$ and alphabet $\{0, \dots, m\}$. We believe investigating systematically the connections between these approaches is a research direction that may lead to interesting conclusions.

%We presented a Bernoulli factory algorithm that can be directly applied to simulate from rational functions. Moreover, it is easily generalizable to multivariate settings where a die or multiple coins are available. Further research may look into applying the same proposed idea to simulate from a wider class of functions, as this could be achieved by allowing different constructions of the underlying Markov chain. Moreover, it would also be interesting to look into efficient ways to implement monotonic CFTP in the multivariate case. 
 
 \section{Acknowledgments}
 We would like to thank anonymous referees for their helpful suggestions that greatly improved the presentation of the paper.
 We thank Susanna Brown, Oliver Johnson and Krzysztof Oleszkiewicz for helpful discussions. K{\L} acknowledges funding form the Royal Society via the University
Research Fellowship scheme. AW has been supported by EPSRC and GM has been supported by EPSRC through the OxWaSP Programme. Finally, AW and K{\L} thank the Warwick Undergraduate Research Scholarship Scheme for supporting the project in its initial stages in the summer of 2013.

\appendix
\section{Background}
\label{sec:appendix_background}
\subsection{Sampling from known categorical distributions} 
\label{sec:appendix_sample_known_distributions}
Sampling from a known distribution $\boldsymbol{\mu} = (\mu_0,\ldots,\mu_k) \in \Delta^k$ is usually done by sampling $U \sim \text{Unif}(0,1)$ and setting
\begin{equation*}
Z = i \quad \text{ if } \quad \sum_{j=0}^{i-1} \mu_j < U \leq \sum_{j=0}^i \mu_j.
\end{equation*}
In the spirit of \cite{Nacu2005,Latuszynski2009}, we can consider $B$, the binary representation of $U$ and notice that this is an iid
sequence of Bern$(1/2)$. Let $B_{1:l}$ denote the first $l$ bits of $U$ and let $(B_{1:l})_{10}$ be its representation in base 10. Clearly $(B_{1:l})_{10} \leq U \leq (B_{1:l})_{10}+2^{-l}$, so that we could set
\begin{equation*}
Z = i \qquad \text{ if } \qquad   \sum_{j=0}^{i-1} \mu_j \leq  (B_{1:l})_{10}
\quad \text{and} \quad (B_{1:l})_{10}+2^{-l} \leq \sum_{j=0}^i \mu_j
\end{equation*}
where $l$ is big enough so that there exists an $i$ such that the condition above is satisfied. 

Therefore, we do not need to have access to a generator of uniform
random variables to sample from a categorical distribution --- it is
enough to obtain a sequence of independent tosses of a fair
coin. Algorithm \ref{alg:fair_coin} is a variation of Von Neumann's
algorithm \cite{VonNeumann1951} that 
outputs a fair coin given access to an iid sequence of rolls of an arbitrary die.

\begin{algorithm}[H]
	\caption{Fair coins from a die}
	\label{alg:fair_coin}
	 \hspace*{\algorithmicindent}\justifying\textbf{Input:} black box to sample from $\boldsymbol{p} \in \Delta^m$. \\
	 \hspace*{\algorithmicindent}\justifying\textbf{Output:} a sample from $\text{Bern}(1/2)$.
	\begin{algorithmic}[1]
		%\State{Fix arbitrary $k \in \{0,\ldots,m-1\}$}
		\State{Sample $X_1, X_2 \stackrel{iid}{\sim} \boldsymbol{p}$}
		\If{$X_1 < X_2$} set $Y := 0$
		\ElsIf{$X_1 > X_2$} set $Y := 1$
		\ElsIf{$X_1 = X_2$}{ discard $X_1, X_2$ and GOTO 1}
		\EndIf
		\State{\textbf{Output}  $Y$}
	\end{algorithmic}
\end{algorithm}

Consequently, given a black box to sample from an unknown distribution
$\boldsymbol{p} \in \Delta^m$, Algorithm
\ref{alg:categorical_distribution} outputs a sample from a known distribution $\boldsymbol{\mu} \in
\Delta^k$. In particular, if $\boldsymbol{\mu}$ is a disaggregation of $\boldsymbol{\nu}$, Algorithm
\ref{alg:categorical_distribution} can be used to obtain a sample $Y \sim \boldsymbol{\mu}$
given $X \sim \boldsymbol{\nu}$ as in equation \eqref{eq:disaggregation}. 
\begin{algorithm}[H]
	\caption{Categorical distribution from a die}
	\label{alg:categorical_distribution}
	 \hspace*{\algorithmicindent}\justifying\textbf{Input:} black box to sample from $\boldsymbol{p} \in \Delta^m$. \\
	 \hspace*{\algorithmicindent}\justifying\textbf{Output:} a sample from a known distribution $\boldsymbol{\mu} \in \Delta^k$.
	\begin{algorithmic}[1]
		\State{Sample $Y \sim \text{Bern}(1/2)$ using Algorithm \ref{alg:fair_coin}}
		\State{If $l = 1$ set $B_{1:1} = Y$, otherwise set $B_{1:l} = B_{1:l-1}|Y$}
		\If{there exists an $i \in \{0,\ldots,k\}$ such that 
			\begin{equation*}
			\sum_{j=0}^{i-1} \mu_j \leq  (B_{1:l})_{10} < (B_{1:l})_{10}+2^{-l} \leq \sum_{j=0}^i \mu_j
			\end{equation*} 
		}{ set $Z := i$}
		\Else
		\State{Set $l = l+1$ and GOTO 2}
		\EndIf
		\State{\textbf{Output}  $Z$}
	\end{algorithmic}
\end{algorithm}

Let $L$ be the random number of loops of Algorithm
\ref{alg:categorical_distribution} before terminating. The computation verifies that $\mathbb{P}(L>l) \leq
k2^{-l}$.

\subsection{Coupling from the past} 
\label{sec:appendix_cftp}

Perfect sampling is a well developed
approach \cite{Huber2016} to devise specialised algorithms,
necessarily with random running time, that will output a single draw
exactly from the stationary distribution of a Markov chain, rather than from
its approximation.   
Coupling From the Past (CFTP) \cite{Propp1996} is a pioneering
technique in this field and illustrative for our purposes. The idea behind the method relies on starting the chain at time $-\infty$, so that at present time one would have a sample from the stationary distribution. This may not seem practical, but as pointed out in \cite{Propp1996}, one can make use of coupled chains to decide when to stop tracking the chain back in time. In practice, it is convenient to introduce an update function for the chain. Given a state $i$ and a source of randomness, the update function returns the state of the chain at the next step. Such source of randomness is commonly represented by a single draw from a uniform random variable $U$, as one could then consider its binary representation to have an arbitrary number of uniform random variables and transform them, at least in principle, via inversion sampling. However, for our specific application, it is natural to consider the given die a source of randomness. As discussed in Section \ref{sec:appendix_sample_known_distributions}, we can resort to the die to generate uniform random variables. However, for better clarity, we consider having access to both the die and a uniform random variable  as source of randomness for the update function, defined as follows:

\begin{defi}[Update function]
\label{def:update_function}
Let $(X_t)_{t \in \mathbb{N}}$ be a Markov chain on $\Omega = \{0,1,\ldots,k\}$. Assume $\boldsymbol{p} \in \Delta^m$ and let $B \sim \boldsymbol{p}$ and $U \sim \text{Unif}(0,1)$. A function
\begin{equation*}
\phi: \Omega \times \{0,1,\ldots,m\} \times [0,1] \rightarrow \Omega
\end{equation*}
is an \emph{update function} for the Markov chain $(X_t)_{t \in \mathbb{N}}$ if
\begin{equation*}
\mathbb{P}(X_{t+1} = j | X_t = i) = \mathbb{P}(\phi(i,B,U) = j), \qquad \forall i,j \in \Omega.
\end{equation*}
\end{defi} 
We will write $\phi_t(x,\boldsymbol{B},\boldsymbol{U}) = \phi(\phi(\ldots(\phi(x,B_1,U_1),B_2,U_2),\ldots),B_t,U_t)$ to indicate the state of the chain after $t$ steps when starting from $x$.

Given an update function $\phi$, CFTP is implementable via Algorithm \ref{alg:CFTP}. Under mild assumptions, namely that there is a positive probability of termination, Algorithm \ref{alg:CFTP} produces samples from the stationary distribution of the Markov chain \cite{Propp1996, Huber2016}.

\begin{algorithm}[H]
	\caption{Coupling From the Past}
	\label{alg:CFTP}
	\hspace*{\algorithmicindent}\justifying\textbf{Input}: an update function
        $\phi$ for a  Markov chain $(X_t)_{t \in \mathbb{N}}$ on
        $\Omega = \{0,\ldots,k\}$ with unique stationary distribution $\pi$;
        a black box to
        sample from $\boldsymbol{p} \in \Delta^m$.
	\\
	\hspace*{\algorithmicindent}\justifying\textbf{Output}: A sample from $\pi$.
	\begin{algorithmic}[1]
		\State{Set $T \gets 1$}
		\InlineFor{$i = 0,\ldots,k$}{$X_0^{(i)} \gets i$}
		\Repeat
			\State{Sample independently $B_{-T} \sim \boldsymbol{p}$ and $U_{-T} \sim \text{Unif}(0,1)$}
			\InlineFor{$i = 0,\ldots,k-1$}{$X_0^{(i)} \gets \phi_T(i,(B_{-T},\ldots,B_{-1}),(U_{-T},\ldots,U_{-1}))$}
			\State{Set $T \gets T+1$}
		\Until{$X_0^{(0)} = X_0^{(1)} = \ldots = X_0^{(k)}$}
		\State{\textbf{Output} $X_0^{(0)}$}
	\end{algorithmic}
\end{algorithm}

Notice that CFTP needs to keep track of the trajectories of $k$
coupled chains. If $k$ is large implementing the algorithm may become
infeasible. A more efficient version of CFTP can be designed for monotonic
Markov chains~\cite{Propp1996}. In particular, assume that the state space $\Omega$ of
the Markov chain $(X_t)_{t \in \mathbb{N}}$ admits a partial order
$\preceq$, and there exist the maximum and the minimum states, say $0$
and $k$, respectively; i.e., $\forall j \in \Omega$, $j \preceq k$ and
$0 \preceq j$. The monotonic update function is defined as follows.
\begin{defi}[Monotonic update function]
An update function $\phi$ as in Definition \ref{def:update_function} is \emph{monotonic} if for all $B \in \{0,1,\ldots,m\}$ and $U \in [0,1]$ 
\begin{equation} \label{eqn:mono}
i \preceq j  \quad \Longrightarrow \quad \phi(i,B,U) \preceq \phi(j,B,U).
\end{equation}
\end{defi}

In the monotonic case it is enough to track coalescence of just two chains,
started from the minimum and the maximum state. Algorithm
\ref{alg:monotonic_CFTP} presents CFTP with monotonic update function $\phi$.

\begin{algorithm}[H]
	\caption{Monotonic Coupling From the Past}
	\label{alg:monotonic_CFTP}
	\hspace*{\algorithmicindent}\justifying\textbf{Input}: a
        monotonic update function $\phi$ for a Markov chain $(X_t)_{t
          \in \mathbb{N}}$ on $\Omega = \{0,\ldots,k\}$ with minimum and maximum states, $0$ and $k$
        respectively and with unique stationary
        distribution $\pi$; a black box to sample from $\boldsymbol{p} \in \Delta^m$.
	\\
	\hspace*{\algorithmicindent}\justifying\textbf{Output}: A sample from $\pi$.
	\begin{algorithmic}[1]
		\State{Set $T \gets 1$, $X_0 \gets 0$ and $Y_0 \gets k$}	
		\Repeat
			\State{Sample independently $B_{-T} \sim \boldsymbol{p}$ and $U_{-T} \sim \text{Unif}(0,1)$}
			\State{Set $X_0 \gets \phi_T(1,(B_{-T},\ldots,B_{-1}),(U_{-T},\ldots,U_{-1}))$}
			\State{Set $Y_0 \gets \phi_T(k,(B_{-T},\ldots,B_{-1}),(U_{-T},\ldots,U_{-1}))$}
			\State{Set $T \gets T+1$}\label{line:time_monotonic_cftp}
		\Until{$X_0 = Y_0$}
		\State{\textbf{Output} $X_0$}
	\end{algorithmic}
\end{algorithm}

\begin{remark}
Commonly, monotonic CFTP is implemented so that $T$ is doubled at each iteration (replacing line \ref{line:time_monotonic_cftp} of the algorithm with $T \gets 2T$). This allows for a binary search of the coalescence time, while trying to minimise the simulation effort. The choice is arbitrary: we decided to keep $T \gets T+1$ as in applications rolling the die may be the most time-consuming step.
\end{remark}

\section{Proofs} \label{sec:appendix_proofs}

Define the scaled by $d$ discrete $m$ dimensional simplex as
\begin{equation*}
\Lambda_d^m = \left\{ \boldsymbol{n} = (n_0,\ldots,n_{m}) \in \{0,1,\ldots,d\}^{m+1}: \sum_{i=0}^{m} n_i = d \right\}.
\end{equation*}

%
%\noindent \textbf{Proof of Proposition \ref{prop:increasing_degree_ladder}}
%\begin{proof}
%We can construct a new probability distribution $\pi'(\boldsymbol{p})$ on $\Omega' = \{1,\ldots,mk\}$ by multiplying each state $i \in \Omega$ with a term $p_j$ for $j \in \{1,\ldots,m\}$.  A state of $\pi'(\boldsymbol{p})$ is then of the form $\pi_i(\boldsymbol{p})p_j$.
%
%To check that $\pi'(\boldsymbol{p})$ is a disaggregation of $\pi(\boldsymbol{p})$, consider $k$ sets $A_1,\ldots,A_k$ defined as 
%\begin{equation*}
%A_i = \{a \in \Omega': \pi'_a(\boldsymbol{p}) = \pi_i(\boldsymbol{p})p_j \text{ for a }j \in \{1,\ldots,m\} \}.
%\end{equation*}
%Then
%\begin{equation*}
%\sum_{a \in A_i} \pi'_a(\boldsymbol{p}) = \sum_{j=1}^m \pi_i(\boldsymbol{p})p_j = \pi_i(\boldsymbol{p})\sum_{j=1}^m p_j = \pi_i(\boldsymbol{p})
%\end{equation*}
%so that Definition \ref{def:disaggregation} is satisfied.
%\end{proof}

%\noindent \textbf{Proof of Proposition \ref{prop:thinning_ladder}}
%\begin{proof}
%Add together all states $i, j \in \Omega$ where the same monomials
%appear, that is if $\boldsymbol{n}_i = \boldsymbol{n}_j$, and call
%this new distribution $\pi'$. Then, $\pi'(\boldsymbol{p})$ is a fine
%multivariate ladder and $\pi(\boldsymbol{p})$ is clearly its
%disaggregation. Clearly, the construction preserves the connectedness.
%\end{proof}

\noindent \textbf{Proof of Proposition \ref{prop:multivariate_impose_connected_fineness_condition}}
\begin{proof}[\unskip\nopunct]
By construction $\pi'$ is a fine ladder on $\Omega' =
\{0,\ldots,w\}$. Since one augmentation operation yields a ladder
sampling from which is equivalent to sampling from $\pi$, so does
d-fold augmentation. It remains to show that $\pi'$ is connected. To
this end note that each state of $\pi'$ is of the form
$C \pi_i(\boldsymbol{p})  {{d}\choose{\boldsymbol{n}}} \prod_{l=0}^{m}
p_l^{n_l}$ for some constant $C$, some $i \in \{0, ..., k\}$ and some $\boldsymbol{n} \in \Lambda_d^m$. Define sets $A_0,\ldots,A_{k}$ as
\begin{equation}
\label{eq:sets_disaggregation_ladder}
A_i = \{a \in \Omega': \pi'_a(\boldsymbol{p}) = C_a\pi_i(\boldsymbol{p})
{{d}\choose{\boldsymbol{n}}} \prod_{l=0}^{m} p_l^{n_l} \text{ for some }
\boldsymbol{n} \in \Lambda_d^m, \; C_a \in \mathbb{R} \}.
\end{equation}
First notice that for a fixed $i$ all the states in $A_i$ are
connected by construction due to $d-$fold augmentation. It is then
enough to show that $A_i \cap A_j \neq \emptyset$ for all $j \neq
i$. Indeed, let $\boldsymbol{n}_i$ and $\boldsymbol{n}_j$ be the
degree of $\pi_i(\boldsymbol{p})$ and $\pi_j(\boldsymbol{p})$
respectively. Then, the numerators of
$\pi_j(\boldsymbol{p}){{d}\choose{\boldsymbol{n}_i}} \prod_{l=0}^{m}
p_l^{n_{i,l}} $ and
$\pi_i(\boldsymbol{p}){{d}\choose{\boldsymbol{n}_j}} \prod_{l=0}^{m}
p_l^{n_{j,l}} $ have the same degree $\boldsymbol{n}_i +
\boldsymbol{n}_j$ and the respective state $a \in \Omega'$ with
probability $\pi_{a'}(\boldsymbol{p})$ of degree $\boldsymbol{n}_i +
\boldsymbol{n}_j$ satisfies $a \in  A_i \cap A_j.$
\end{proof}

\begin{lemma}[P\'olya \cite{Polya1928}]
\label{lemma:polya}
Let $f: \Delta^{m} \rightarrow \mathbb{R}$ be a homogeneous and positive polynomial in the variables $p_0,\ldots,p_{m}$, i.e. all the monomials of the polynomial have the same degree. Then, for all sufficiently large $n$, all the coefficients of $(p_0+\ldots+p_{m})^nf(p_0,\ldots,p_{m})$ are positive.
\end{lemma}
\begin{lemma}
\label{lemma:decomposition}
Let $f: \Delta^{m} \rightarrow (0,1)$ be a rational function over $\mathbb{R}$. Then, there exist homogeneous polynomials 
\begin{align*}
d(\boldsymbol{p}) &= d(p_0,\ldots,p_{m}) = \sum_{\boldsymbol{n} \in \Lambda_d^m} d_{\boldsymbol{n}} \prod_{j=0}^{m} p_j^{n_j}, \\ e(\boldsymbol{p}) &= e(p_0,\ldots,p_{m}) = \sum_{\boldsymbol{n} \in \Lambda_d^m} e_{\boldsymbol{n}} \prod_{j=0}^{m} p_j^{n_j},
\end{align*}
where $d_{\boldsymbol{n}}$ and $e_{\boldsymbol{n}}$ are real coefficients such that $0 \leq d_{\boldsymbol{n}} \leq e_{\boldsymbol{n}}$ and $f(\boldsymbol{p}) = d(\boldsymbol{p})/e(\boldsymbol{p})$. We will refer to $d$ as the degree of the decomposition.
\end{lemma}
\begin{proof}
The lemma is a variation of Lemma 2.7 of \cite{Mossel2005}, where $m=1$ and coefficients are integers, and the proof follows the reasoning therein. 

As $f(\boldsymbol{p})$ is a rational function, it may be written as 
\begin{equation*}
f(\boldsymbol{p}) = \frac{\overline{D}(\boldsymbol{p})}{\overline{E}(\boldsymbol{p})},
\end{equation*}
and we can assume that $\overline{D}(\boldsymbol{p})$ and $\overline{E}(\boldsymbol{p})$ are relatively prime polynomials. Since $f(\boldsymbol{p}) \in (0,1)$ for all $\boldsymbol{p} \in \Delta^{m}$ and $\overline{D}(\boldsymbol{p})$ does not share any common root with $\overline{E}(\boldsymbol{p})$, it follows that $\overline{D}(\boldsymbol{p})$ and $\overline{E}(\boldsymbol{p})$ do not change sign in $\Delta^{m}$ so that we can assume without loss of generality that $\overline{D}(\boldsymbol{p})$ and $\overline{E}(\boldsymbol{p})$ are positive polynomials. Let $d_0$ be the maximum degree of the polynomials $\overline{D}(\boldsymbol{p})$ and $\overline{E}(\boldsymbol{p})$. A general representation of the polynomials is given by
\begin{equation*}
\overline{D}(\boldsymbol{p}) = \sum_{i=0}^d \sum_{\boldsymbol{n} \in \Lambda^m_i} a_{\boldsymbol{n}} \prod_{j=0}^{m} p_j^{n_j}, \qquad  \overline{E}(\boldsymbol{p}) = \sum_{i=0}^d \sum_{\boldsymbol{n} \in \Lambda^m_i} b_{\boldsymbol{n}} \prod_{j=0}^{m} p_j^{n_j}.
\end{equation*}
Notice that in general $\overline{D}(\boldsymbol{p})$ and
$\overline{E}(\boldsymbol{p})$ are not homogeneous polynomials, but it
is possible to increase the degree of each term of the summation to be equal to $d_0$. In
fact, since $p_0+\ldots+p_{m} = 1$, one can use the multinomial theorem to define homogeneous polynomials $D(\boldsymbol{p})$ and $E(\boldsymbol{p})$ as
\begin{align*}
\overline{D}(\boldsymbol{p}) &= \sum_{i=0}^d \sum_{\boldsymbol{n} \in \Lambda^m_i} a_{\boldsymbol{n}}(p_0+\ldots+p_{m})^{d-i} \prod_{j=0}^{m} p_j^{n_j} \\
&= \sum_{i=0}^d \sum_{\boldsymbol{n} \in \Lambda^m_i} \sum_{\boldsymbol{n}' \in \Lambda_{d-i}^m} a_{\boldsymbol{n}} {{d-i}\choose{\boldsymbol{n}'}} \prod_{j=0}^{m} p_j^{n_j + n'_j} \\
&=  \sum_{\boldsymbol{n} \in \Lambda_d^m} d_{\boldsymbol{n}} \prod_{j=0}^{m} p_j^{n_j} =: D(\boldsymbol{p}),
\end{align*}
where 
\begin{equation*}
d_{\boldsymbol{n}} = \sum_{i=0}^d \sum_{\boldsymbol{\tilde{n}} \in \Lambda^m_i} \sum_{\boldsymbol{n}' \in \Lambda_{d-i}^m: \boldsymbol{\tilde{n}}+\boldsymbol{n}' = \boldsymbol{n}} a_{\boldsymbol{\tilde{n}}} {{d-i}\choose{\boldsymbol{n}'}}.
\end{equation*}
Analogously

\begin{align*}
\overline{E}(\boldsymbol{p}) &= \sum_{i=0}^d \sum_{\boldsymbol{n} \in \Lambda^m_i} b_{\boldsymbol{n}}(p_0+\ldots+p_{m})^{d-i} \prod_{j=0}^{m} p_j^{n_j} =  \sum_{\boldsymbol{n} \in \Lambda_d^m} e_{\boldsymbol{n}} \prod_{j=0}^{m} p_j^{n_j} := E(\boldsymbol{p}).
\end{align*}
Notice that $D(\boldsymbol{p})$ and $E(\boldsymbol{p})$ are positive polynomials. Moreover, since $f(\boldsymbol{p}) < 1$, it follows that also $E(\boldsymbol{p}) - D(\boldsymbol{p})$ is a positive polynomial. Therefore, by Lemma \ref{lemma:polya} there exists a sufficiently large $n,$ such that the polynomials $d(\boldsymbol{p}) = (p_0+\ldots+p_{m})^nD(\boldsymbol{p}), \;$ $e(\boldsymbol{p}) = (p_0+\ldots+p_{m})^nE(\boldsymbol{p})$ and $e(\boldsymbol{p}) - d(\boldsymbol{p}),$ all have positive coefficients. Hence, as required, $0 \leq d_{\boldsymbol{n}} \leq e_{\boldsymbol{n}}$ and 
\begin{equation}\label{eq:frac_of_f}
f(\boldsymbol{p}) = \frac{\overline{D}(\boldsymbol{p})}{\overline{E}(\boldsymbol{p})} = \frac{D(\boldsymbol{p})}{E(\boldsymbol{p})} = \frac{d(\boldsymbol{p})}{e(\boldsymbol{p})}.
\end{equation}
The degree of the decomposition is therefore $d=d_0+n$.
\end{proof}

\noindent \textbf{Proof of Theorem \ref{thm:decomposition_multivariate}}
\begin{proof}[\unskip\nopunct]
Since $f(\boldsymbol{p}) = (f_0(\boldsymbol{p}),\ldots,f_{v}(\boldsymbol{p}))$ is a rational function, we can apply Lemma \ref{lemma:decomposition} to each $f_i(\boldsymbol{p})$ and write
\begin{equation*}
f(\boldsymbol{p}) = \left( \frac{d_0(\boldsymbol{p})}{e_0(\boldsymbol{p})}, \frac{d_1(\boldsymbol{p})}{e_1(\boldsymbol{p})},\ldots,\frac{d_{v}(\boldsymbol{p})}{e_{v}(\boldsymbol{p})} \right).
\end{equation*}
Let $C(\boldsymbol{p})$ be the lowest common multiple of the denominators $e_i(\boldsymbol{p})$ and express $f(\boldsymbol{p})$ as 
\begin{equation*}
f(\boldsymbol{p}) = \frac{1}{C(\boldsymbol{p})}(g_0(\boldsymbol{p}),\ldots,g_{v}(\boldsymbol{p})).
\end{equation*}
Assume w.l.o.g. that each polynomial $g_i(\boldsymbol{p})$ has degree $d$ (if this is not the case, let $d_i$ be the degree of $g_i(\boldsymbol{p})$ and multiply it by $(p_0+\ldots +p_{m})^{d-d_i}$) and write
\begin{equation}
\label{eq:thm_decomposition_homogeneous}
\frac{g_i(\boldsymbol{p})}{C(\boldsymbol{p})} = \frac{1}{C(\boldsymbol{p})} \sum_{\boldsymbol{n} \in \Lambda_d^m} a_{i,\boldsymbol{n}} \prod_{j=0}^{m} p_j^{n_j}.
\end{equation}
Having applied Lemma \ref{lemma:decomposition} it follows $a_{i,\boldsymbol{n}} \geq 0$ for all $i \in \{0,\ldots,v\}$, $\boldsymbol{n} \in \Lambda_d^m$. Therefore, we can construct a distribution $\pi': \Delta^m
\rightarrow \Delta^{w}$ on $\Omega' = \{0,\ldots,w\}$, where $w < (v+1){{d+m}\choose{m}}$ and where each state is one
term of the summation in \eqref{eq:thm_decomposition_homogeneous} for
a fixed $i$ and thus of the form $\frac{1}{C(\boldsymbol{p})} a_{i,\boldsymbol{n}} \prod_{j=0}^{m-1}
p_j^{n_j}$. 

By construction $\pi'$ is a disaggregation of $f$. Indeed, consider $v$ sets $A_0,\ldots,A_{v}$ defined as
\begin{equation*}
A_i = \{a \in \Omega': \pi'_a(\boldsymbol{p}) = \frac{1}{C(\boldsymbol{p})} a_{i,\boldsymbol{n}} \prod_{j=0}^{m} p_j^{n_j} \text{ for a } \boldsymbol{n} \in \Lambda_d^m \}.
\end{equation*}
It then follows
\begin{equation*}
f_i(\boldsymbol{p}) = \frac{g_i(\boldsymbol{p})}{C(\boldsymbol{p})} = \sum_{h \in A_i} \pi'_h(\boldsymbol{p}) = \frac{1}{C(\boldsymbol{p})} \sum_{\boldsymbol{n} \in \Lambda_d^m} a_{i,\boldsymbol{n}} \prod_{j=0}^{m} p_j^{n_j}.
\end{equation*}
 By discarding any null term in $\pi'(\boldsymbol{p})$,
it follows that $\pi'$ is a multivariate ladder. Finally, via
Proposition
\ref{prop:multivariate_impose_connected_fineness_condition} we
construct a fine and connected multivariate ladder $\pi: \Delta^m \rightarrow \Delta^k$ where $k < \min\{(w+1)(m+1)^d, {{2d+m}\choose{m}}\}$, such that
sampling from each $f$, $\pi'$ and $\pi$ is equivalent.

\end{proof}

\noindent \textbf{Proof of Proposition \ref{prop:multivariate_markov_chain}}
\begin{proof}[\unskip\nopunct]
We shall prove the result by showing that $P$ is a stochastic matrix and that the detailed balance condition is satisfied for all $\boldsymbol{p} \in \Delta^m$. Recall that the off-diagonal elements of $P$ are given by the off-diagonal elements of $V \circ W$ where $\circ$ denotes the entrywise product, $W$ is defined in equation \eqref{eq:transition_matrix_W} and $V$ is the output of Algorithm \ref{alg:construction_markov_chain}. We first prove that
\begin{equation*}
\sum_{j \in \mathcal{N}_b(i)} V_{i,j} \leq 1, \quad \forall b \in \{0,\ldots,m\}, i \in \Omega.
\end{equation*}
Notice that by how the weights $\mathcal{W}_b(i)$ are defined within the algorithm, we have $\sum_{j \in \mathcal{N}_b(i)} V_{i,j} = \mathcal{W}_b(i)$.

Having fixed $i$ and $b$, assume that one of the $V_{i,j}$, where $j \in \mathcal{N}_b(i)$, is obtained in line \ref{alg:line:max} of the algorithm. Denote by $\mathcal{W}^\star_b(i)$ the new value of $\mathcal{W}_b(i)$ after it has been updated for all $j \in \mathcal{N}_b(i)$. It follows
\begin{align*}
\mathcal{W}^\star_b(i) = \mathcal{W}_b(i) + \sum_{j \in \mathcal{N}_b(i)} \frac{R_j}{\mathcal{S}_b(i)}  = \mathcal{W}_b(i) + \sum_{j \in \mathcal{N}_b(i)} \frac{R_j}{\sum_{h \in \mathcal{N}_b(i)} R_h} (1-\mathcal{W}_b(i)) = 1,
\end{align*}
where the value of $S_b(i)$ is given in line 9 of the algorithm.
At this point the algorithm has assigned a value to $V_{i,j}$ for all $j \in \mathcal{N}_b(i)$ and thus $\sum_{j \in \mathcal{N}_b(i)} V_{i,j} = \mathcal{W}^\star_b(i) = 1$.

Assume now that all the $V_{i,j}$ for $j \in \mathcal{N}_b(i)$ have been assigned in line \ref{alg:line:reverse} of the algorithm. For fixed $i$, we then have that $j \in \mathcal{N}_b(i)$ and let $d \in \{0,\ldots,m\}$ such that $i \in \mathcal{N}_d(j)$. Then $V_{j,i}$ is assigned in line \ref{alg:line:max} of the algorithm. Denote the new value of $\mathcal{W}_b(i)$ assigned in line \ref{alg:line:reverse} of the algorithm as $\mathcal{W}^\star_b(i)$. It follows
\begin{equation*}
\mathcal{W}^\star_b(i) = \mathcal{W}_b(i) + \frac{R_j}{\mathcal{S}_d(j)} \leq \mathcal{W}_b(i) + \frac{R_j}{\mathcal{S}_b(i)} \leq \mathcal{W}_b(i) + \sum_{j \in \mathcal{N}_b(i)} \frac{R_j}{\mathcal{S}_b(i)} = 1,
\end{equation*}
where the fact that $\mathcal{S}_d(j) \geq \mathcal{S}_b(i)$ follows from the fact that $b$ and $i$ are chosen in line \ref{alg:line_max_arg} of the algorithm to maximise $\mathcal{S}_b(i)$. The value of $\mathcal{W}_b(i)$ will then always be less or equal than 1, so that $\sum_{j \in \mathcal{N}_b(i)} V_{i,j} = \mathcal{W}_b(i) \leq 1$.

We then have
\begin{equation*}
\sum_{\substack{j=0 \\ j\neq i}}^{k} P_{i,j} = \sum_{b=0}^{m} \sum_{j \in \mathcal{N}_b(i)} V_{i,j}p_b \leq \sum_{b=0}^{m} p_b = 1,
\end{equation*}
as required. It is now enough to prove that $\pi(\boldsymbol{p})$ satisfies the detailed balance condition for all $\boldsymbol{p} \in \Delta^m$. If $j \not\in \mathcal{N}(i)$, then $P_{i,j} = P_{j,i} = 0$ and the
balance condition is trivially satisfied. For $j \in \mathcal{N}(i)$
we have
\begin{equation*}
\frac{\pi_j(\boldsymbol{p})}{\pi_i(\boldsymbol{p})} = \frac{ R_j \prod_{h=0}^{m} p_h^{n_{j,h}}}{ R_i  \prod_{h=0}^{m} p_h^{n_{i,h}}} 
= \frac{R_jp_b}{R_ip_c},
\end{equation*}
and by equation \eqref{eq:transition_matrix_W}, $W_{i,j}/W_{j,i} = p_b/p_c$. The fact that $V_{i,j}/V_{j,i} = R_j/R_i$ follows directly from how these values are assigned in the algorithm for the pair $i, j$ in lines \ref{alg:line:max} and \ref{alg:line:reverse}. Given the connectedness condition, $\pi(\boldsymbol{p})$ is also the unique limiting distribution.
\end{proof}

\noindent \textbf{Proof of Proposition \ref{prop:Peskun_optimal}}
\begin{proof}[\unskip\nopunct]
By contradiction, assume that there exists a different reversible Markov chain with transition matrix $Q$ that has the same adjacency structure and stationary distribution as the $P$-chain, and such that $Q \succeq_P P$. It follows that also $Q$ has a similar decomposition as in equation \eqref{eq:transition_matrix_product} and the off-diagonal elements of $Q$ will be the same as the entries of $\tilde{V} \circ W$, where $\circ$ denotes the entrywise product and with $W$ as in equation \eqref{eq:transition_matrix_W}, while $\tilde{V}$ is a matrix of real numbers. Since $Q \succeq_P P$ and $Q \neq P$, there must exist indices $i,j$ such that $\tilde{V}_{i,j} > V_{i,j}$. We distinguish two cases:
\begin{itemize}
\item The value of $V_{i,j}$ is assigned in line \ref{alg:line:max} of Algorithm \ref{alg:construction_markov_chain}. Then, let $b \in \{0,\ldots,m\}$ such that $j \in \mathcal{N}_b(i)$ and notice that by how the algorithm is designed we have $\sum_{j \in \mathcal{N}_b(i)} V_{i,j} = 1$ (cf. proof of Proposition \ref{prop:multivariate_markov_chain}). Therefore $\sum_{j \in \mathcal{N}_b(i)} \tilde{V}_{i,j} > 1$. We reach a contradiction by observing
\begin{equation*}
\sum_{\substack{j=0 \\ j\neq i}}^{k-1} Q_{i,j} = \sum_{c=0}^{m} \sum_{j \in \mathcal{N}_c(i)} \tilde{V}_{i,j}p_c \xrightarrow{p_b \rightarrow 1} \sum_{j \in \mathcal{N}_b(i)} \tilde{V}_{i,j} > 1.
\end{equation*}
\item The value of $V_{i,j}$ is assigned in line \ref{alg:line:reverse} of Algorithm \ref{alg:construction_markov_chain}. Since the $Q$-chain is reversible, it follows that also $\tilde{V}_{j,i} > V_{j,i}$. However, the value of $V_{j,i}$ is assigned in line \ref{alg:line:max} of the algorithm and we reach the same contradiction as before.
\end{itemize}
\end{proof}

\noindent \textbf{Proof of Proposition \ref{prop:multivariate_update_function}}
\begin{proof}[\unskip\nopunct]
Fix a state $i \in \Omega$ and notice that if $j \not\in \mathcal{N}(i)$, then $\mathbb{P}(\phi(i,B,U) = j) = P_{i,j} = 0$. For any outcome $b \in
\{0,\ldots,m\}$ on the die, recall $\mathcal{N}_b(i) =
\{j_0,\ldots,j_{w}\},$ is  the set of states accessible from $i$. It follows for any $j_l \in \mathcal{N}_b(i)$ that
\begin{align*}
\mathbb{P}(\phi(i,B,U) = j_l) &= \mathbb{P}\left(B = b, \sum_{h=0}^{l-1} V_{i,j_h} < U \leq \sum_{h=0}^l V_{i,j_h} \right) \\
&= p_b\mathbb{P}\left(U \leq V_{i,j_l}\right) = P_{i,j_l}.
\end{align*}
Hence, $\phi$ is an update function for the Markov chain $(X_t)_{t \in \mathbb{N}}$.
\end{proof}

\noindent \textbf{Proof of Corollary \ref{prop:markov_chain_update_function_univariate}}
%\textcolor{teal}{GM: new proof}
\begin{proof}[\unskip\nopunct]
Given a fine and connected ladder $\pi: (0,1) \rightarrow \Delta^k$ as in equation \eqref{eq:univariate_connected_fine_ladder}, for $1 \leq i \leq k-1,$ we have 
\begin{alignat*}{2}
&\mathcal{N}_0(i) = \{i-1\}, \qquad &&\mathcal{N}_1(i) = \{i+1\}, \qquad \mathcal{N}(i) = \{i-1,i+1\}, \\
&\mathcal{S}_0(i) = R_{i-1}, \qquad &&\mathcal{S}_1(i) = R_{i+1}.
\end{alignat*}
Then, the matrix $W$ defined in equation \eqref{eq:transition_matrix_W} and the off-diagonal entries of the matrix $V$ output by Algorithm \ref{alg:construction_markov_chain} are given by
\begin{equation*}
W_{i,j} = \begin{cases}
p &\quad \text{if } j = i+1 \\
(1-p) &\quad \text{if } j = i-1 \\
0 &\quad \text{otherwise }
\end{cases}, \qquad 
V_{i,j} = \begin{cases}
\frac{R_{i+1}}{R_i \vee R_{i+1}} &\quad \text{if } j = i+1 \\
\frac{R_{i-1}}{R_{i-1} \vee R_i} &\quad \text{if } j = i-1 \\
0 &\quad \text{otherwise }
\end{cases}
\end{equation*}
Therefore, the transition matrix $P$ defined in \eqref{eq:transition_matrix_product} is equivalent to \eqref{eq:transition_probs_ladder_univariate} and the update function defined in \eqref{eq:update_function_ladder_multivariate} is the same as \eqref{eq:update_function_ladder_univariate}.

To see why $\phi$ is a monotonic update function, consider $i \leq j$. It is trivial to check that $\phi(i,B,U) \leq \phi(j,B,U)$ if $j \neq i+1$. If $j = i+1$, the monotonic condition would not be satisfied only if $\phi(i,B,U) = i+1$ and $\phi(i+1,B,U) = i$. However, this can not happen as it would require $B$ to be equal to $0$ and $1$ simultaneously.
\end{proof}

\noindent \textbf{Proof of Theorem \ref{conj:logconcavity_convolution_binomial}}
\begin{proof}

Let $w_0,\ldots, w_{n_0}$ be the probabilities of $W$ on $\Omega = \{0,\ldots,n_0\}$. We shall consider generating function $P(x)=\sum_{i=0}^{n_0} w_i x^i$. This function is a product of linear and quadratic functions, that is
\[
	P(x) = c \prod_{j=1}^{k_0} ((x-a_j)^2 + b_j^2) \prod_{l=1}^{l_0} (x+c_l),
\]
where $b_j \ne 0$ and $c_l>0$ (the latter follows from the fact that a
polynomial with positive coefficients cannot have positive
roots). Now, it suffices to show that for big $n$ the sequences of
coefficients generated by
\[
	Q_n(x) = ((x-a)^2 + b^2) (1+x)^n, \quad L_n(x) = (x+c)(1+x)^n, \enspace \text{where }b\neq 0, c > 0, 
\] 
is positive and log concave. Indeed, since convolution preserves positivity and log-concavity, and corresponds to summing random variables, we can find suitable binomial $B(n,1/2)$ (whose generating function is precisely $\frac{1}{2^n}(1+x)^n$) for each factor of $P$ separately. Note that we ignore normalizing constants, as positivity and log-concavity are not affected.

\vspace{0.2cm}

The rest is just an attempt to verify this. In case of $L_n$ there is nothing to prove since the sequence generated by $x+c$ with $c > 0$ is $(c,1,0,\ldots)$ and it is positive and log-concave itself. Since $(x-a)^2 + b^2 = x^2 -2ax + a^2+b^2$, the sequence generated by $Q_n$ is 
\begin{equation*}
	a_k =  (a^2+b^2){n \choose k} - 2a  {n \choose k-1} + {n \choose k-2} , \qquad k \geq 0.
\end{equation*}
Here we adapt the notation ${n \choose k}=0$ for $k<0$ and $k>n$.
We first show that for big $n$ this sequence is non-negative. The inequality $a_k \geq 0$ is equivalent to
\[
	(a^2+b^2)(n-k+1)(n-k+2) -2a k (n-k+2) + k(k-1) \geq 0.
\]
Let us treat the left hand side as a polynomial in $k$. This is
\begin{align*}
	k^2 \left(a^2+2 a+b^2+1\right) & +k \left((-2 n-3) \left(a^2
   +b^2\right)-2 a n-4 a-1 \right) \\ & \qquad \qquad \qquad \qquad +(n+1) (n+2) \left(a^2+b^2\right).
\end{align*}
Since the coefficient in front of $k^2$ is positive, we can hope to find $n$ such that this polynomial is positive for all real $k$. For this the $\Delta$ of this quadratic form should be negative. We have
\begin{align*}
	\Delta & = \left((-2 n-3) \left(a^2+b^2\right)-2 a n-4
                 a-1\right)^2
\\ & \qquad \qquad -4 (n+1) (n+2) \left(a^2+b^2\right) \left(a^2+2 a+b^2+1\right) \\
	& = - 4 b^2 n^2 +  4 (a + 2 a^2 + a^3 - 2 b^2 + a b^2) n 
\\ & \qquad \qquad 
+ (1 + 8 a + 14 a^2 + 8 a^3 + a^4 - 2 b^2 + 8 a b^2 + 2 a^2 b^2 + 
   b^4)  .
\end{align*}
As we can see the leading term is $-4b^2n^2$ and so for big $n$ we get $\Delta<0$.

We now show that for big $n$ the sequence $a_k$ is strictly log-concave; i.e., $a_k^2 > a_{k+1}a_{k-1}$. This is trivially true for $k=0$ and $k=n+2$, but it is also easily verified for $k\in \{1, n-1, n, n+1\}$ by just substituting the value of $k$ and letting $n \to \infty$.

To prove the result for $k \in \{2,\ldots,n-2\}$, rewrite the coefficients $a_k$ as:
\[
	a_k = {n \choose k-1} \left[ \frac{n-k+1}{k}(a^2+b^2) + \frac{k-1}{n-k+2} - 2a \right].
\] 
The inequality $a_k^2 > a_{k+1}a_{k-1}$ reduces to 
\begin{align*}
	{n \choose k-1}^2 & \left[ \frac{n-k+1}{k}(a^2+b^2) +
                          \frac{k-1}{n-k+2} - 2a \right]^2  
\\ &
                                                                    \hspace{1.7cm}
                                                                     >
     {n
                                                                    \choose
                                                                    k}
                                                                    \left[
                                                                    \frac{n-k}{k+1}(a^2+b^2)
                                                                    +
                                                                    \frac{k}{n-k+1}
                                                                    -
                                                                    2a
                                                                    \right]                                                           
  \\
	& \hspace{2cm} \times {n \choose k-2} \left[ \frac{n-k+2}{k-1}(a^2+b^2) + \frac{k-2}{n-k+3} - 2a \right],
\end{align*}
This is
\begin{align*}
	\frac{k}{k-1} \cdot \frac{n-k+2}{n-k+1} & \left[ \frac{n-k+1}{k}(a^2+b^2) + \frac{k-1}{n-k+2} - 2a \right]^2 \\
	& \hspace{1.7cm}
	 >  \left[ \frac{n-k}{k+1}(a^2+b^2) + \frac{k}{n-k+1} -
          2a \right] \\ & \hspace{2cm} \times \left[ \frac{n-k+2}{k-1}(a^2+b^2) + \frac{k-2}{n-k+3} - 2a \right].
\end{align*}
%
%Equivalently,
%\begin{align*}
%&	 \left[ \frac{n-k+1}{k}(a^2+b^2) + \frac{k-1}{n-k+2} - 2a \right]^2 \\
%	& \hspace{1.7cm}
%	 >  \left[ \frac{(n-k)(n-k+1)}{k(k+1)}(a^2+b^2) + 1 - 2a \frac{n-k+1}{k} \right] \\
%	& \hspace{2cm} \times \left[ (a^2+b^2) + \frac{(k-2)(k-1)}{(n-k+3)(n-k+2)} - 2a \frac{k-1}{n-k+2} \right].
%\end{align*}
To deal with it we rewrite it slightly.
\begin{align*}
&	 \left[ \frac{n-k+1}{k}(a^2+b^2) + \frac{k-1}{n-k+2} - 2a \right]^2 \\
	& \hspace{0.3cm}
	 >  \left[ \frac{(n-k)}{(k+1)}(a^2+b^2) + \frac{k}{n-k+1} - 2a  \right] \\
	 & \qquad \times \left[ \frac{n-k+1}{k} (a^2+b^2) + \frac{(k-2)(k-1)(n-k+1)}{(n-k+3)(n-k+2)k}\right. \\
	 &\qquad \qquad \left. - 2a \frac{(k-1)(n-k+1)}{(n-k+2)k} \right].
\end{align*}
For big $n$ and fixed $a, b$, the right hand side is a product of two positive factors. We shall take the square root of both sides and use the inequality $2\sqrt{xy} \leq x+y$ to bound the right hand side. Then, it is enough to verify:
\begin{align*}
&	 2\left[\frac{n-k+1}{k}(a^2+b^2) + \frac{k-1}{n-k+2} - 2a\right] \\
	& \hspace{0.3cm}
	 >  \left[ \frac{(n-k)}{(k+1)}(a^2+b^2) + \frac{k}{n-k+1} - 2a  \right] \\
	 & \qquad + \left[ \frac{n-k+1}{k} (a^2+b^2) + \frac{(k-2)(k-1)(n-k+1)}{(n-k+3)(n-k+2)k}\right. \\
	 &\qquad \qquad \left. - 2a \frac{(k-1)(n-k+1)}{(n-k+2)k} \right].
\end{align*}
Rewrite it by taking the RHS to the LHS and collecting common factors.
\begin{align*}
& \left[\frac{(n-k+1)}{k} - \frac{n-k}{k+1}\right](a^2+b^2) + \frac{2(k-1)}{n-k+2} - \frac{k}{n-k+1}  \\
& \quad - \frac{(k-2)(k-1)(n-k+1)}{(n-k+3)(n-k+2)k} - \left[ 1 - \frac{(k-1)(n-k+1)}{(n-k+2)k} \right]2a > 0.
\end{align*}
Notice:
\begin{align*}
& \bullet \quad \frac{(n-k+1)}{k} - \frac{n-k}{k+1} = \frac{n+1}{k(k+1)}, \\
& \bullet \quad \frac{2(k-1)}{n-k+2} - \frac{k}{n-k+1}  - \frac{(k-2)(k-1)(n-k+1)}{(n-k+3)(n-k+2)k} =  \\
& \quad \qquad -\frac{(1+k+(k-1)^2+n-(k-1) n)(n+1)}{(n-k+1)(n-k+2)(n-k+3)k}, \\
& \bullet \quad 1 - \frac{(k-1)(n-k+1)}{(n-k+2)k} = \frac{n+1}{(n-k+2)k}.
\end{align*}

 Thus, by taking the common denominator, it is enough to verify $P_{a,b}(k,n) > 0$, where
\begin{align*}
&P_{a,b}(k,n) = (a^2+b^2) (n-k+3) (n-k+2) (n-k+1) \\
&\hspace{0.5cm} - \left(1+k+(k-1)^2+n-(k-1) n\right)(k+1) \\
&\qquad - 2a(k+1)(n-k+3)(n-k+1).
\end{align*}
This is a polynomial of degree three in $k$. The discriminant of a cubic polynomial $Ak^3 + Bk^2 + Ck + D$ is given by
\begin{equation*}
\Delta = B^2C^2 - 4 A C^3 - 4B^3D - 27A^2D^2 + 18 ABCD,
\end{equation*}
and is negative if there are two conjugate complex and one real roots. 

In our case the discriminant of $k \to P_{a,b}(k,n)$ is
\begin{equation*}
\Delta(n,a,b) = -4b^2n^6+O(n^5),
\end{equation*}
and so for big $n$ it is negative (recall that $b \neq 0$). We conclude that there is only one real root. Notice that
\begin{align*}
P_{a,b}(2,n) &= (a^2+b^2)n^3+O(n^2), \\
P_{a,b}(n-2,n) &= n^2+O(n),
\end{align*}
so that for $n$ big enough, $P_{a,b}(k,n) > 0$ for all $k \in [2,n-2]$ as desired.

\end{proof}

\noindent \textbf{Proof of Proposition \ref{prop:bound_expected_rolls}}
\begin{proof}
Augment the ladder $d$ times to construct a new ladder $\pi':
\Delta^m  \rightarrow \Delta^w$, where $w < \min\{(k+1)(m+1)^d, {{2d+m}\choose{m}}\}$. We showed in Proposition \ref{prop:multivariate_impose_connected_fineness_condition} that $\pi'$ is a fine and connected ladder and that we can define sets $A_0,\ldots,A_{k}$ as in equation \eqref{eq:sets_disaggregation_ladder}. We now show that for any state $a \in \Omega'$, it is always possible to move to a different state if $b \in E$ is rolled (except from the state proportional to $p_b^{2d}/C(\boldsymbol{p})$). Fix a state $a \in \Omega'$ in the set $A_i$, therefore of the form
\begin{equation*}
C_a\pi_i(\boldsymbol{p}){{d}\choose{\boldsymbol{n}}} \boldsymbol{p}^{\boldsymbol{n}}, \quad\text{for some }\boldsymbol{n} \in \Lambda_d^m.
\end{equation*}
If $\boldsymbol{p}^{\boldsymbol{n}} \neq p_b^d$, then there exists another state $a' \in A_i$ connected to $a$ and such that $n'_{a',b} = n'_{a,b}+1$ and the chain may move to it. We showed in the proof of Proposition \ref{prop:multivariate_impose_connected_fineness_condition} that $A_i \cap A_j \neq \emptyset, \forall j \neq i$. Therefore, if $\boldsymbol{p}^{\boldsymbol{n}} = p_b^d$ there exists a connected state $a'$ in $A_j \neq A_i$ such that $n'_{a',b} = n'_{a,b}+1$, unless $\pi_i(\boldsymbol{p}) \propto p_b^{2d}/C(\boldsymbol{p})$.

Now, consider applying CFTP on the ladder $\pi'$ using the transition matrix of Proposition \ref{prop:multivariate_markov_chain} and the update function of Proposition \ref{prop:multivariate_update_function}. We prove the bound by considering sets of moves that, regardless of the starting point, end up in a singleton. Let $a$ be the minimum of the entries of the matrix $V$, as produced by Algorithm \ref{alg:construction_markov_chain}. This choice of $a$ allows us to conclude that whenever we draw $U < a$ in the CFTP algorithm and $B \in E$, then all the tracked particles move, except the particles in the state proportional to $p_b^{2d}/C(\boldsymbol{p})$. Therefore if such event happens on $2d$ consecutive iterations, then the algorithm necessarily ends as all the particles must have coalesced in the state proportional to $p_{b}^{2d}/C(\boldsymbol{p})$. That is, if $u_1 \leq a, \ldots, u_{2d} \leq a$ we can write
\begin{equation*}
\phi_{2d}(i,(b,\ldots,b),(u_1,\ldots,u_{2d})) = \{a\}, \qquad \forall i \in \{0,\ldots,w\},
\end{equation*}
where $a \in \Omega'$ is the state of the ladder proportional to $p_b^{2d}/C(\boldsymbol{p})$. Let $\tau_b$ be the number of iterations required for this event to happen for the first time. The probability generating function of $\tau_b$ is given by
\begin{align*}
f_{\tau_b}(x) &= \sum_{j=0}^\infty (ap_b)^{2d}x^{2d}\left[ (1-ap_b)x+\ldots+(ap_b)^{2d-1}(1-ap_b)x^{2d} \right]^j \\
 &= \frac{x^{2 d} (a p_b)^{2 d} (a p_b x-1)}{a p_b x (a p_b x)^{2 d}-x (a p_b x)^{2 d}+x-1},
\end{align*}
so that
\begin{equation*}
\mathbb{E}[\tau_b] = f'_{\tau_b}(1) = \frac{(ap_b)^{-2d}-1}{1-ap_b}.
\end{equation*}
Since the number of required rolls $N$ equals the number of iterations of the algorithm, it follows that $N \leq \tau_b$. The same reasoning holds for all $b \in E$, so that we conclude:
\begin{equation*}
\mathbb{E}[N] \leq \min_{b \in E} \mathbb{E}[\tau_b] = \min_{b \in E} \frac{(ap_b)^{-2d}-1}{1-ap_b}.
\end{equation*}
\end{proof}

\noindent \textbf{Proof of Corollary \ref{cor:bound_expected_tosses}}
\begin{proof}
Follows by Proposition \ref{prop:bound_expected_rolls} by noticing that in the case $m = 1$, we necessarily have $E = \{0,1\}$.
\end{proof}

\noindent \textbf{Proof of Proposition \ref{prop:efficiency_monotonic_CFTP}}
\begin{proof}
Requiring $\pi$ to be strictly log-concave is equivalent to have $R_i^2 > R_{i-1}R_{i+1}$ for all $i \in \{1,\ldots,k-1\}$ by equation \eqref{eq:univariate_connected_fine_ladder}. In turn, this implies
\begin{equation}
\label{eq:condition_fast_simulation}
\frac{R_{i}}{R_{i-1} \vee R_{i}} \geq \frac{R_{i+1}}{R_{i} \vee R_{i+1}}, \quad
\frac{R_{i}}{R_i \vee R_{i+1}} \geq \frac{R_{i-1}}{R_{i-1} \vee R_{i}}, 
\end{equation}
so that $\rho \leq 1$ since $P_{i,i+1} \geq P_{i+1,i+2}$ and $P_{i+1,i} \geq P_{i,i-1}$. However, given $p \in (0,1)$, it cannot be that $\rho = 1$. Indeed, this could happen only if $P_{i,i+1} = P_{i+1,i+2}$ and $P_{i+1,i} = P_{i,i-1}$. However, this would imply either $R_i^2=R_{i-1}R_{i+1}$ or $R_{i+1}^2 = R_iR_{i+2}$ thus contradicting strict log-concavity. We then conclude that $\rho \in (0,1)$ for all $p \in (0,1)$.

Denote by $X_t^i$ the chain at time $t$ given that it started in state $i$. Monotonic CFTP (cf. Algorithm \ref{alg:monotonic_CFTP}) tracks backwards in time the trajectories of the coupled chains $X_t^0$ and $X_t^k$ and stops when the two coalesce. Following the notation of \cite{Propp1996}, let $T_\star$ be the time this happens and, to ease the analysis, define $T^\star$ as the smallest time such that $X_t^0 = X_t^k$, where the chains are now tracked forwards in time. Notice that $T_\star$ and $T^\star$ have the same distribution and that the number of tosses $N$ required by the algorithm equals $T_\star$. 

Define $D_t^{i,j} = |X_t^i - X_t^j|$ as the distance between two coupled  particles started at states $i$ and $j$ after $t$ steps. In particular, focus on the distance $D_t^{i,i+1}$ between two particles started at consecutive states. At each step a $p$-coin is tossed and a uniform random variable is drawn so that the trajectories of the two chains can be tracked in a coupled fashion. In particular, given equation \eqref{eq:condition_fast_simulation}, we have that the two particles started at states $i$ and $(i+1)$ can in one step either stay still, coalesce in state $i$ or state $(i+1)$, move to states $(i+1)$ and $(i+2)$ or move to states $(i-1)$ and $i$ respectively. Therefore, after one step the distance between the two coupled and consecutive particles can either decrease by 1 or remain the same:
\begin{equation*}
D_1^{i,i+1} = \begin{cases}
0 &\quad \text{with probability } (P_{i,i+1}-P_{i+1,i+2}) + (P_{i+1,i}-P_{i,i-1})\\
1 &\quad \text{with probability } 1-(P_{i,i+1}-P_{i+1,i+2}) - (P_{i+1,i}-P_{i,i-1})\\
\end{cases}
\end{equation*}
where the transition probabilities $P_{i,j}$ are given in equation \eqref{eq:transition_probs_ladder_univariate}. Denote  by
\begin{equation*}
\rho_{i,i+1} = 1-(P_{i,i+1}-P_{i+1,i+2}) - (P_{i+1,i}-P_{i,i-1}),
\end{equation*}
 so that $\mathbb{E}[D_1^{i,i+1}]  = \rho_{i,i+1}$. Let $\rho =
 \max_i \rho_{i,i+1}$ and notice that by conditioning on how the particles move on the first step and by the Markov property, it follows
 \begin{align*}
\mathbb{E}[D_t^{i,i+1}] &= P_{i+1,i+2}\mathbb{E}[D_{t-1}^{i+1,i+2}] + P_{i,i-1}\mathbb{E}[D_{t-1}^{i-1,i}] + \left(1 - P_{i,i+1} - P_{i+1,i}\right)\mathbb{E}[D_{t-1}^{i,i+1}] \\
&\leq (\mathbb{E}[D_{t-1}^{i-1,i}] \vee \mathbb{E}[D_{t-1}^{i,i+1}] \vee \mathbb{E}[D_{t-1}^{i+1,i+2}]) \rho_{i,i+1} \nonumber \\ 
&\leq \rho^{t}, \nonumber
\end{align*}
where $\vee$ denotes the maximum between two numbers. 

To conclude, notice that $\mathbb{P}(T^\star \geq t) = \mathbb{P}(D_t^{0,k} \geq 1)$. It then follows by Markov's inequality and the result above that
\begin{align*}
\mathbb{P}(T^\star \geq t) = \mathbb{P}(D_t^{0,k} \geq 1) \leq \mathbb{E}[D_t^{0,k}] = \sum_{i=0}^{k-2} \mathbb{E}[D_t^{i,i+1}] \leq (k-1) \rho^t,
\end{align*}
as desired.
\end{proof}

\noindent \textbf{Proof of Lemma \ref{lemma:impose_efficiency_condition}}
\begin{proof}
Note that a univariate ladder is log-concave if its coefficients $R_i$ define a log-concave sequence. Then, let $R$ be a random variable on $\{0,\ldots,k\}$ having p.m.f. proportional to the coefficients $R_i$ of the ladder $\pi$, that is such that $\mathbb{P}(R=i) \propto R_i$. Moreover, let $n$ be such that $Z = R + B_n$ is strictly log-concave, as stated in Theorem \ref{conj:logconcavity_convolution_binomial}. Consider $\pi':
(0,1) \rightarrow \Delta^{k+n}$, an $n$-fold augmentation of $\pi$. As
noticed in Remark \ref{rmk:augmenting_convolution}, $Y \sim \pi'(p)$  has the same distribution as $\pi + \text{Bin}(n,p)$ and the coefficients $R'_i$s of the ladder $\pi'$ are proportional to $\mathbb{P}(Z=i)$. The desired result holds by noticing that multiplication by a constant preserves log-concavity. 
\end{proof}

\bibliographystyle{imsart-number}
\bibliography{references_dice_enterprise}

\begin{thebibliography}{42}
% BibTex style file: imsart-number.bst, 2017-11-03
% Default style options (sort=1,type=number).
% Used options (sort=1,type=number).

\bibitem{asmussen1992stationarity}
\begin{barticle}[author]
\bauthor{\bsnm{Asmussen},~\bfnm{S{\o}ren}\binits{S.}},
  \bauthor{\bsnm{Glynn},~\bfnm{Peter~W}\binits{P.~W.}} \AND
  \bauthor{\bsnm{Thorisson},~\bfnm{Hermann}\binits{H.}}
(\byear{1992}).
\btitle{Stationarity detection in the initial transient problem}.
\bjournal{ACM Transactions on Modeling and Computer Simulation (TOMACS)}
\bvolume{2}
\bpages{130--157}.
\end{barticle}
\endbibitem

\bibitem{Blanchet2005}
\begin{btechreport}[author]
\bauthor{\bsnm{Blanchet},~\bfnm{J}\binits{J.}} \AND
  \bauthor{\bsnm{Meng},~\bfnm{X}\binits{X.}}
(\byear{2005}).
\btitle{Exact sampling, regeneration and minorization conditions}
\btype{Technical Report},
\bpublisher{Tech. rep., Columbia University. URL http://www. columbia. edu/\~{}
  b2814/papers/JSMsent. pdf}.
\end{btechreport}
\endbibitem

\bibitem{Blanchet2017}
\begin{barticle}[author]
\bauthor{\bsnm{{Blanchet}},~\bfnm{J.}\binits{J.}} \AND
  \bauthor{\bsnm{{Zhang}},~\bfnm{F.}\binits{F.}}
(\byear{2017}).
\btitle{{Exact Simulation for Multivariate It\^o Diffusions}}.
\bjournal{ArXiv e-prints}.
\end{barticle}
\endbibitem

\bibitem{Bubley1997}
\begin{binproceedings}[author]
\bauthor{\bsnm{Bubley},~\bfnm{R.}\binits{R.}} \AND
  \bauthor{\bsnm{Dyer},~\bfnm{M.}\binits{M.}}
(\byear{1997}).
\btitle{Path coupling: A technique for proving rapid mixing in Markov chains}.
In \bbooktitle{Proceedings 38th Annual Symposium on Foundations of Computer
  Science}
\bpages{223-231}.
\bdoi{10.1109/SFCS.1997.646111}
\end{binproceedings}
\endbibitem

\bibitem{cai2019efficient}
\begin{barticle}[author]
\bauthor{\bsnm{Cai},~\bfnm{Yang}\binits{Y.}},
  \bauthor{\bsnm{Oikonomou},~\bfnm{Argyris}\binits{A.}},
  \bauthor{\bsnm{Velegkas},~\bfnm{Grigoris}\binits{G.}} \AND
  \bauthor{\bsnm{Zhao},~\bfnm{Mingfei}\binits{M.}}
(\byear{2019}).
\btitle{An Efficient $\varepsilon$-BIC to BIC Transformation and Its
  Application to Black-Box Reduction in Revenue Maximization}.
\bjournal{arXiv preprint arXiv:1911.10172}.
\end{barticle}
\endbibitem

\bibitem{Karp2000}
\begin{barticle}[author]
\bauthor{\bsnm{Dagum},~\bfnm{Paul}\binits{P.}},
  \bauthor{\bsnm{Karp},~\bfnm{Richard}\binits{R.}},
  \bauthor{\bsnm{Luby},~\bfnm{Michael}\binits{M.}} \AND
  \bauthor{\bsnm{Ross},~\bfnm{Sheldon}\binits{S.}}
(\byear{2000}).
\btitle{An optimal algorithm for {M}onte {C}arlo estimation}.
\bjournal{SIAM J. Comput.}
\bvolume{29}
\bpages{1484--1496}.
\bdoi{10.1137/S0097539797315306}
\bmrnumber{1744833}
\end{barticle}
\endbibitem

\bibitem{Dale2015}
\begin{barticle}[author]
\bauthor{\bsnm{Dale},~\bfnm{Howard}\binits{H.}},
  \bauthor{\bsnm{Jennings},~\bfnm{David}\binits{D.}} \AND
  \bauthor{\bsnm{Rudolph},~\bfnm{Terry}\binits{T.}}
(\byear{2015}).
\btitle{Provable quantum advantage in randomness processing}.
\bjournal{Nature communications}
\bvolume{6}
\bpages{8203}.
\end{barticle}
\endbibitem

\bibitem{Dughmi2017}
\begin{bincollection}[author]
\bauthor{\bsnm{Dughmi},~\bfnm{Shaddin}\binits{S.}},
  \bauthor{\bsnm{Hartline},~\bfnm{Jason~D.}\binits{J.~D.}},
  \bauthor{\bsnm{Kleinberg},~\bfnm{Robert}\binits{R.}} \AND
  \bauthor{\bsnm{Niazadeh},~\bfnm{Rad}\binits{R.}}
(\byear{2017}).
\btitle{Bernoulli factories and black-box reductions in mechanism design}.
In \bbooktitle{S{TOC}'17---{P}roceedings of the 49th {A}nnual {ACM} {SIGACT}
  {S}ymposium on {T}heory of {C}omputing}
\bpages{158--169}.
\bpublisher{ACM, New York}.
\bmrnumber{3678179}
\end{bincollection}
\endbibitem

\bibitem{Fill1998}
\begin{barticle}[author]
\bauthor{\bsnm{Fill},~\bfnm{James~Allen}\binits{J.~A.}}
(\byear{1998}).
\btitle{An interruptible algorithm for perfect sampling via {M}arkov chains}.
\bjournal{Ann. Appl. Probab.}
\bvolume{8}
\bpages{131--162}.
\bdoi{10.1214/aoap/1027961037}
\bmrnumber{1620346}
\end{barticle}
\endbibitem

\bibitem{flajolet2011buffon}
\begin{binproceedings}[author]
\bauthor{\bsnm{Flajolet},~\bfnm{Philippe}\binits{P.}},
  \bauthor{\bsnm{Pelletier},~\bfnm{Maryse}\binits{M.}} \AND
  \bauthor{\bsnm{Soria},~\bfnm{Mich{\`e}le}\binits{M.}}
(\byear{2011}).
\btitle{On Buffon machines and numbers}.
In \bbooktitle{Proceedings of the twenty-second annual ACM-SIAM symposium on
  Discrete algorithms}
\bpages{172--183}.
\bpublisher{Society for Industrial and Applied Mathematics}.
\end{binproceedings}
\endbibitem

\bibitem{Flegal2012}
\begin{barticle}[author]
\bauthor{\bsnm{Flegal},~\bfnm{James~M.}\binits{J.~M.}} \AND
  \bauthor{\bsnm{Herbei},~\bfnm{Radu}\binits{R.}}
(\byear{2012}).
\btitle{Exact sampling for intractable probability distributions via a
  {B}ernoulli factory}.
\bjournal{Electron. J. Stat.}
\bvolume{6}
\bpages{10--37}.
\bdoi{10.1214/11-EJS663}
\bmrnumber{2879671}
\end{barticle}
\endbibitem

\bibitem{Goncalves2017b}
\begin{barticle}[author]
\bauthor{\bsnm{{Gon{\c c}alves}},~\bfnm{F.~B.}\binits{F.~B.}},
  \bauthor{\bsnm{{{\L}atuszy{\'n}ski}},~\bfnm{K.~G.}\binits{K.~G.}} \AND
  \bauthor{\bsnm{{Roberts}},~\bfnm{G.~O.}\binits{G.~O.}}
(\byear{2017}).
\btitle{{Exact Monte Carlo likelihood-based inference for jump-diffusion
  processes}}.
\bjournal{ArXiv e-prints}.
\end{barticle}
\endbibitem

\bibitem{Goncalves2017}
\begin{barticle}[author]
\bauthor{\bsnm{Gon\c{c}alves},~\bfnm{Fl\'avio~B.}\binits{F.~B.}},
  \bauthor{\bsnm{{{\L}atuszy{\'n}ski}},~\bfnm{Krzysztof}\binits{K.}} \AND
  \bauthor{\bsnm{Roberts},~\bfnm{Gareth~O.}\binits{G.~O.}}
(\byear{2017}).
\btitle{Barker's algorithm for {B}ayesian inference with intractable
  likelihoods}.
\bjournal{Braz. J. Probab. Stat.}
\bvolume{31}
\bpages{732--745}.
\bdoi{10.1214/17-BJPS374}
\bmrnumber{3738176}
\end{barticle}
\endbibitem

\bibitem{MR2958668}
\begin{barticle}[author]
\bauthor{\bsnm{Goyal},~\bfnm{Vineet}\binits{V.}} \AND
  \bauthor{\bsnm{Sigman},~\bfnm{Karl}\binits{K.}}
(\byear{2012}).
\btitle{On simulating a class of {B}ernstein polynomials}.
\bjournal{ACM Trans. Model. Comput. Simul.}
\bvolume{22}
\bpages{Art. 12, 5}.
\bdoi{10.1145/2133390.2133396}
\bmrnumber{2958668}
\end{barticle}
\endbibitem

\bibitem{MR1957197}
\begin{barticle}[author]
\bauthor{\bsnm{Henderson},~\bfnm{Shane~G.}\binits{S.~G.}} \AND
  \bauthor{\bsnm{Glynn},~\bfnm{Peter~W.}\binits{P.~W.}}
(\byear{2003}).
\btitle{Nonexistence of a class of variate generation schemes}.
\bjournal{Oper. Res. Lett.}
\bvolume{31}
\bpages{83--89}.
\bdoi{10.1016/S0167-6377(02)00217-1}
\bmrnumber{1957197}
\end{barticle}
\endbibitem

\bibitem{Herbei2014}
\begin{barticle}[author]
\bauthor{\bsnm{Herbei},~\bfnm{Radu}\binits{R.}} \AND
  \bauthor{\bsnm{Berliner},~\bfnm{L.~Mark}\binits{L.~M.}}
(\byear{2014}).
\btitle{Estimating ocean circulation: an {MCMC} approach with approximated
  likelihoods via the {B}ernoulli factory}.
\bjournal{J. Amer. Statist. Assoc.}
\bvolume{109}
\bpages{944--954}.
\bdoi{10.1080/01621459.2014.914439}
\bmrnumber{3265667}
\end{barticle}
\endbibitem

\bibitem{Hoggar1974}
\begin{barticle}[author]
\bauthor{\bsnm{Hoggar},~\bfnm{S.~G.}\binits{S.~G.}}
(\byear{1974}).
\btitle{Chromatic polynomials and logarithmic concavity}.
\bjournal{J. Combinatorial Theory Ser. B}
\bvolume{16}
\bpages{248--254}.
\bdoi{10.1016/0095-8956(74)90071-9}
\bmrnumber{0342424}
\end{barticle}
\endbibitem

\bibitem{MR2784483}
\begin{barticle}[author]
\bauthor{\bsnm{Holtz},~\bfnm{Olga}\binits{O.}},
  \bauthor{\bsnm{Nazarov},~\bfnm{Fedor}\binits{F.}} \AND
  \bauthor{\bsnm{Peres},~\bfnm{Yuval}\binits{Y.}}
(\byear{2011}).
\btitle{New coins from old, smoothly}.
\bjournal{Constr. Approx.}
\bvolume{33}
\bpages{331--363}.
\bdoi{10.1007/s00365-010-9108-5}
\bmrnumber{2784483}
\end{barticle}
\endbibitem

\bibitem{Huber2013}
\begin{barticle}[author]
\bauthor{\bsnm{Huber},~\bfnm{Mark}\binits{M.}}
(\byear{2016}).
\btitle{Nearly optimal {B}ernoulli factories for linear functions}.
\bjournal{Combin. Probab. Comput.}
\bvolume{25}
\bpages{577--591}.
\bdoi{10.1017/S0963548315000371}
\bmrnumber{3506427}
\end{barticle}
\endbibitem

\bibitem{Huber2017}
\begin{barticle}[author]
\bauthor{\bsnm{Huber},~\bfnm{Mark}\binits{M.}}
(\byear{2017}).
\btitle{Optimal linear {B}ernoulli factories for small mean problems}.
\bjournal{Methodol. Comput. Appl. Probab.}
\bvolume{19}
\bpages{631--645}.
\bdoi{10.1007/s11009-016-9518-3}
\bmrnumber{3649562}
\end{barticle}
\endbibitem

\bibitem{Huber2016}
\begin{bbook}[author]
\bauthor{\bsnm{Huber},~\bfnm{Mark~L.}\binits{M.~L.}}
(\byear{2016}).
\btitle{Perfect simulation}.
\bseries{Monographs on Statistics and Applied Probability}
\bvolume{148}.
\bpublisher{CRC Press, Boca Raton, FL}.
\bmrnumber{3443710}
\end{bbook}
\endbibitem

\bibitem{MR3319143}
\begin{barticle}[author]
\bauthor{\bsnm{Jacob},~\bfnm{Pierre~E.}\binits{P.~E.}} \AND
  \bauthor{\bsnm{Thiery},~\bfnm{Alexandre~H.}\binits{A.~H.}}
(\byear{2015}).
\btitle{On nonnegative unbiased estimators}.
\bjournal{Ann. Statist.}
\bvolume{43}
\bpages{769--784}.
\bdoi{10.1214/15-AOS1311}
\bmrnumber{3319143}
\end{barticle}
\endbibitem

\bibitem{Johnson2006}
\begin{barticle}[author]
\bauthor{\bsnm{Johnson},~\bfnm{Oliver}\binits{O.}} \AND
  \bauthor{\bsnm{Goldschmidt},~\bfnm{Christina}\binits{C.}}
(\byear{2006}).
\btitle{Preservation of log-concavity on summation}.
\bjournal{ESAIM Probab. Stat.}
\bvolume{10}
\bpages{206--215}.
\bdoi{10.1051/ps:2006008}
\bmrnumber{2219340}
\end{barticle}
\endbibitem

\bibitem{Keane1994}
\begin{barticle}[author]
\bauthor{\bsnm{Keane},~\bfnm{MS}\binits{M.}} \AND
  \bauthor{\bsnm{O'Brien},~\bfnm{George~L}\binits{G.~L.}}
(\byear{1994}).
\btitle{A Bernoulli factory}.
\bjournal{ACM Transactions on Modeling and Computer Simulation (TOMACS)}
\bvolume{4}
\bpages{213--219}.
\end{barticle}
\endbibitem

\bibitem{Latuszynski2009}
\begin{barticle}[author]
\bauthor{\bsnm{{{\L}atuszy{\'n}ski}},~\bfnm{Krzysztof}\binits{K.}},
  \bauthor{\bsnm{Kosmidis},~\bfnm{Ioannis}\binits{I.}},
  \bauthor{\bsnm{Papaspiliopoulos},~\bfnm{Omiros}\binits{O.}} \AND
  \bauthor{\bsnm{Roberts},~\bfnm{Gareth~O.}\binits{G.~O.}}
(\byear{2011}).
\btitle{Simulating events of unknown probabilities via reverse time
  martingales}.
\bjournal{Random Structures Algorithms}
\bvolume{38}
\bpages{441--452}.
\bdoi{10.1002/rsa.20333}
\bmrnumber{2829311}
\end{barticle}
\endbibitem

\bibitem{Lee2014}
\begin{barticle}[author]
\bauthor{\bsnm{{Lee}},~\bfnm{A.}\binits{A.}},
  \bauthor{\bsnm{{Doucet}},~\bfnm{A.}\binits{A.}} \AND
  \bauthor{\bsnm{{{\L}atuszy{\'n}ski}},~\bfnm{K.}\binits{K.}}
(\byear{2014}).
\btitle{{Perfect simulation using atomic regeneration with application to
  Sequential Monte Carlo}}.
\bjournal{ArXiv e-prints}.
\end{barticle}
\endbibitem

\bibitem{Mendo2016}
\begin{barticle}[author]
\bauthor{\bsnm{Mendo},~\bfnm{Luis}\binits{L.}}
(\byear{2019}).
\btitle{An asymptotically optimal {B}ernoulli factory for certain functions
  that can be expressed as power series}.
\bjournal{Stochastic Process. Appl.}
\bvolume{129}
\bpages{4366--4384}.
\bdoi{10.1016/j.spa.2018.11.017}
\bmrnumber{4013865}
\end{barticle}
\endbibitem

\bibitem{Mira2001}
\begin{barticle}[author]
\bauthor{\bsnm{Mira},~\bfnm{Antonietta}\binits{A.}}
(\byear{2001}).
\btitle{Ordering and improving the performance of {M}onte {C}arlo {M}arkov
  chains}.
\bjournal{Statist. Sci.}
\bvolume{16}
\bpages{340--350}.
\bdoi{10.1214/ss/1015346319}
\bmrnumber{1888449}
\end{barticle}
\endbibitem

\bibitem{Mossel2005}
\begin{barticle}[author]
\bauthor{\bsnm{Mossel},~\bfnm{Elchanan}\binits{E.}} \AND
  \bauthor{\bsnm{Peres},~\bfnm{Yuval}\binits{Y.}}
(\byear{2005}).
\btitle{New coins from old: computing with unknown bias}.
\bjournal{Combinatorica}
\bvolume{25}
\bpages{707--724}.
\bnote{With an appendix by Christopher Hillar}.
\bdoi{10.1007/s00493-005-0043-1}
\bmrnumber{2199432}
\end{barticle}
\endbibitem

\bibitem{Nacu2005}
\begin{barticle}[author]
\bauthor{\bsnm{Nacu},~\bfnm{{\c{S}}erban}\binits{{\c{S}}.}} \AND
  \bauthor{\bsnm{Peres},~\bfnm{Yuval}\binits{Y.}}
(\byear{2005}).
\btitle{Fast simulation of new coins from old}.
\bjournal{Ann. Appl. Probab.}
\bvolume{15}
\bpages{93--115}.
\bdoi{10.1214/105051604000000549}
\bmrnumber{2115037}
\end{barticle}
\endbibitem

\bibitem{niazadeh2017algorithms}
\begin{bphdthesis}[author]
\bauthor{\bsnm{Niazadeh},~\bfnm{Rad}\binits{R.}}
(\byear{2017}).
\btitle{PhD Thesis: Algorithms Vs. Mechanisms.: Mechanism Design for Complex
  Environments},
\btype{PhD thesis},
\bpublisher{Cornell University}.
\end{bphdthesis}
\endbibitem

\bibitem{patel2019experimental}
\begin{barticle}[author]
\bauthor{\bsnm{Patel},~\bfnm{Raj~B}\binits{R.~B.}},
  \bauthor{\bsnm{Rudolph},~\bfnm{Terry}\binits{T.}} \AND
  \bauthor{\bsnm{Pryde},~\bfnm{Geoff~J}\binits{G.~J.}}
(\byear{2019}).
\btitle{An experimental quantum Bernoulli factory}.
\bjournal{Science advances}
\bvolume{5}
\bpages{eaau6668}.
\end{barticle}
\endbibitem

\bibitem{Peskun1973}
\begin{barticle}[author]
\bauthor{\bsnm{Peskun},~\bfnm{P.~H.}\binits{P.~H.}}
(\byear{1973}).
\btitle{Optimum {M}onte-{C}arlo sampling using {M}arkov chains}.
\bjournal{Biometrika}
\bvolume{60}
\bpages{607--612}.
\bdoi{10.1093/biomet/60.3.607}
\bmrnumber{0362823}
\end{barticle}
\endbibitem

\bibitem{Polya1928}
\begin{barticle}[author]
\bauthor{\bsnm{P{\'o}lya},~\bfnm{Georg}\binits{G.}}
(\byear{1928}).
\btitle{{\"U}ber positive darstellung von polynomen}.
\bjournal{Vierteljschr. Naturforsch. Ges. Z{\"u}rich}
\bvolume{73}
\bpages{141--145}.
\end{barticle}
\endbibitem

\bibitem{Powers2001}
\begin{barticle}[author]
\bauthor{\bsnm{Powers},~\bfnm{Victoria}\binits{V.}} \AND
  \bauthor{\bsnm{Reznick},~\bfnm{Bruce}\binits{B.}}
(\byear{2001}).
\btitle{{A new bound for P{\'{o}}lya's Theorem with applications to polynomials
  positive on polyhedra}}.
\bjournal{Journal of Pure and Applied Algebra}.
\bdoi{10.1016/S0022-4049(00)00155-9}
\end{barticle}
\endbibitem

\bibitem{Propp1996}
\begin{binproceedings}[author]
\bauthor{\bsnm{Propp},~\bfnm{James~Gary}\binits{J.~G.}} \AND
  \bauthor{\bsnm{Wilson},~\bfnm{David~Bruce}\binits{D.~B.}}
(\byear{1996}).
\btitle{Exact sampling with coupled {M}arkov chains and applications to
  statistical mechanics}.
In \bbooktitle{Proceedings of the {S}eventh {I}nternational {C}onference on
  {R}andom {S}tructures and {A}lgorithms ({A}tlanta, {GA}, 1995)}
\bvolume{9}
\bpages{223--252}.
\bdoi{10.1002/(SICI)1098-2418(199608/09)9:1/2<223::AID-RSA14>3.3.CO;2-R}
\bmrnumber{1611693}
\end{binproceedings}
\endbibitem

\bibitem{Saumard2014}
\begin{barticle}[author]
\bauthor{\bsnm{Saumard},~\bfnm{Adrien}\binits{A.}} \AND
  \bauthor{\bsnm{Wellner},~\bfnm{Jon~A.}\binits{J.~A.}}
(\byear{2014}).
\btitle{Log-concavity and strong log-concavity: a review}.
\bjournal{Stat. Surv.}
\bvolume{8}
\bpages{45--114}.
\bdoi{10.1214/14-SS107}
\bmrnumber{3290441}
\end{barticle}
\endbibitem

\bibitem{schmon2019bernoulli}
\begin{barticle}[author]
\bauthor{\bsnm{Schmon},~\bfnm{Sebastian~M}\binits{S.~M.}},
  \bauthor{\bsnm{Doucet},~\bfnm{Arnaud}\binits{A.}} \AND
  \bauthor{\bsnm{Deligiannidis},~\bfnm{George}\binits{G.}}
(\byear{2019}).
\btitle{Bernoulli Race Particle Filters}.
\bjournal{arXiv preprint arXiv:1903.00939}.
\end{barticle}
\endbibitem

\bibitem{Sison1995}
\begin{barticle}[author]
\bauthor{\bsnm{Sison},~\bfnm{Cristina~P.}\binits{C.~P.}} \AND
  \bauthor{\bsnm{Glaz},~\bfnm{Joseph}\binits{J.}}
(\byear{1995}).
\btitle{Simultaneous Confidence Intervals and Sample Size Determination for
  Multinomial Proportions}.
\bjournal{Journal of the American Statistical Association}
\bvolume{90}
\bpages{366-369}.
\bdoi{10.1080/01621459.1995.10476521}
\end{barticle}
\endbibitem

\bibitem{dootika2019}
\begin{btechreport}[author]
\bauthor{\bsnm{Vats},~\bfnm{Dootika}\binits{D.}},
  \bauthor{\bsnm{Gon\c{c}alves},~\bfnm{Fl\'avio~B.}\binits{F.~B.}},
  \bauthor{\bsnm{{\L}atuszy\'nski},~\bfnm{Krzysztof}\binits{K.}} \AND
  \bauthor{\bsnm{Roberts},~\bfnm{Gareth~O.}\binits{G.~O.}}
(\byear{2020}).
\btitle{Efficient Bernoulli factory MCMC for intractable posteriors}
\btype{Technical Report}.
\end{btechreport}
\endbibitem

\bibitem{VonNeumann1951}
\begin{bincollection}[author]
\bauthor{\bsnm{{von Neumann}},~\bfnm{John}\binits{J.}}
(\byear{1951}).
\btitle{Various techniques used in connection with random digits}.
In \bbooktitle{Monte Carlo Method}
(\beditor{\bfnm{A.~S.}\binits{A.~S.}~\bsnm{Householder}},
  \beditor{\bfnm{G.~E.}\binits{G.~E.}~\bsnm{Forsythe}} \AND
  \beditor{\bfnm{H.~H.}\binits{H.~H.}~\bsnm{Germond}}, eds.)
\bpages{36--38}.
\bpublisher{National Bureau of Standards Applied Mathematics Series, 12},
  \baddress{Washington, D.C.: U.S. Government Printing Office}.
\end{bincollection}
\endbibitem

\bibitem{yuan2016experimental}
\begin{barticle}[author]
\bauthor{\bsnm{Yuan},~\bfnm{Xiao}\binits{X.}},
  \bauthor{\bsnm{Liu},~\bfnm{Ke}\binits{K.}},
  \bauthor{\bsnm{Xu},~\bfnm{Yuan}\binits{Y.}},
  \bauthor{\bsnm{Wang},~\bfnm{Weiting}\binits{W.}},
  \bauthor{\bsnm{Ma},~\bfnm{Yuwei}\binits{Y.}},
  \bauthor{\bsnm{Zhang},~\bfnm{Fang}\binits{F.}},
  \bauthor{\bsnm{Yan},~\bfnm{Zhaopeng}\binits{Z.}},
  \bauthor{\bsnm{Vijay},~\bfnm{R}\binits{R.}},
  \bauthor{\bsnm{Sun},~\bfnm{Luyan}\binits{L.}} \AND
  \bauthor{\bsnm{Ma},~\bfnm{Xiongfeng}\binits{X.}}
(\byear{2016}).
\btitle{Experimental quantum randomness processing using superconducting
  qubits}.
\bjournal{Physical review letters}
\bvolume{117}
\bpages{010502}.
\end{barticle}
\endbibitem

\end{thebibliography}

\end{document}